\documentclass[reqno]{amsart}
\usepackage[top=2cm,bottom=2cm,right=2.5cm,left=2.5cm]{geometry}
\usepackage{amssymb}
\usepackage{amsmath, amsthm, amscd, amsfonts, amssymb, graphicx, color}
\usepackage[bookmarksnumbered, colorlinks, plainpages]{hyperref}
\hypersetup{colorlinks=true,linkcolor=red, anchorcolor=green, citecolor=cyan, urlcolor=red, filecolor=magenta, pdftoolbar=true}
 
\usepackage{hyperref}

\newtheorem{theorem}{Theorem}[section]
\newtheorem{lemma}[theorem]{Lemma}

\newtheorem{corollary}[theorem]{Corollary}
\theoremstyle{definition}
\newtheorem{definition}[theorem]{Definition}
\newtheorem{example}[theorem]{Example}

\theoremstyle{remark}

\numberwithin{equation}{section}

\begin{document}
	
	\setcounter{page}{1}
	
	\title{New fixed point theorems for $(\phi, F)-$contraction on rectangular b-metric spaces}
	
	\author{ Mohamed Rossafi$^{1*}$ and Abdelkarim Kari$^{2}$}
	
	\address{$^{1}$LaSMA Laboratory Department of Mathematics Faculty of Sciences, Dhar El Mahraz University Sidi Mohamed Ben Abdellah, B. P. 1796 Fes Atlas, Morocco}
	\email{\textcolor[rgb]{0.00,0.00,0.84}{rossafimohamed@gmail.com; mohamed.rossafi@usmba.ac.ma}}
	\address{$^{2}$AMS Laboratory Faculty of Sciences Ben M’Sik, Hassan II University, Casablanca, Morocco}
	\email{\textcolor[rgb]{0.00,0.00,0.84}{abdkrimkariprofes@gmail.com}}
	
	
	
	\date{
		\newline \indent $^{*}$Corresponding author}

	\begin{abstract}
		The Banach contraction principle is the most celebrated fixed point theorem, it has been generalized in various directions.
		In this paper, inspired by the concept of $(\phi, F)-$contraction in metric spaces, introduced by Wardowski.
		We present the notion of $(\phi, F)-$contraction in  $b-$rectangular metric spaces to study the existence and uniqueness of fixed point for the mappings in this spaces. Our results improve many existing results.
	\end{abstract}
	\maketitle
	\section{Introduction }
	The Banach contraction principle is an important result in the theory of metric spaces \cite{BA}. Many mathematician studied a lot of interesting extensions and generalizations, (see \cite{BRO,KAN,RE} and the recent works of Wardowski in \cite{WAR,WARD1}. Piri and Kumam introduced a new type of contractions called F-contraction \cite{PIRI} and  F-weak contraction \cite{PIRI1} and proved a new fixed point theorem concerning F-contractions. His valuable work has been elaborated via generalizing the Banach contraction principle. As a consequence of those generalizations so many contraction were introduced $ w $-distance \cite{R. Batra 2020R}, generalized weakly 
contraction mappings \cite{WK}, rational contractions \cite{H.I.J}, Caristi’s theorem \cite{KIR}. Huge work have been done in this direction.
	
	A well-known generalization of metric spaces are $b-$metric spaces were introduced by Czerwik \cite{CZ}, many mathematicians worked on this interesting space. For more, the reader can refer to \cite{CJS, H, HH}.
	
In 2014, Jleli et al. in \cite{JS, JSA} analysed a generalization of the Banach fixed point theorem in a new type of contraction mappings on metric spaces called $\theta-$contraction. In 2018, Wardowski \cite{WAR2} analysed a generalization of the Banach fixed point theorem in a new type of contraction mappings on metric spaces called $(\phi, F)-$contraction.

	Very recently Kari et al. \cite{KRMA} extended Wardowski’s ideas to the case of nonlinear $F-$contraction via $w-$distance and studied the solution of certain integral equations under a suitable set of hypotheses.
	 
In 2000, for the first time generalized metric spaces were introduced by Branciari \cite{BRA}, in such a way that triangle inequality is replaced by the quadrilateral inequality  
	\begin{equation*}
		d(x,y)\leq d(x,z)+d(z,u)+d(u,y), 
	\end{equation*}  
	for all pairwise distinct points $ x,y,z $ and $ u $. Any metric space is a generalized metric space but in general, generalized metric space might not be a metric space. Various fixed point results were established on such spaces, (see \cite{AZ,JS,SM} and references therein).

In 2015 George et al. \cite{RG} announced the notion of $b-$rectangular metric space, many authors initiated and studied a lot of existing fixed point theorems in such spaces, (see \cite{DIN,DINI,KARO,KAR,RO}).

	Very recently, Kari et al. \cite{KAR} introduced the notion of $\theta-\phi-$contraction in $b-$rectangular metric spaces and proved a fixed point theorem for $\theta-\phi-$contraction in $ b-$rectangular metric spaces.
	
	Motivated by the results of Wardowski in \cite{WAR} and of  Kari et al. \cite{KARO}, we establish in this paper a fixed point result for  $(\phi, F)-$contraction in the setting of $b-$rectangular metric spaces. The results presented in the paper extend the corresponding results of Kannan \cite{KAN} and Reich \cite{RE} on $b-$rectangular metric space.
	\section{preliminaries}
	\begin{definition}
		\cite{RG}. Let $X$ be a nonempty set, $s\geq 1$ be a given real number, and let
		
		d: $X\times X\rightarrow \left[ 0,+\infty \right[ $
		be a mapping such that for all $x,y$ $\in X$ and all distinct points $u,v\in
		X,$ each  distinct from $x$ and $y$:
		\begin{itemize}
			\item[1.] $d\left(x, y\right) =0,$ if only if $x=y;$
			\item[2.] $d\left(x, y\right) =d\left(y, x\right);$
			\item[3.] $d\left(x, y\right) \leq $ $s\left[d\left( x,u\right)+d\left(u, v\right) +d\left(v, y\right) \right] $ $\left( b-rectan gular\ inequality \right) .$
		\end{itemize}
		Then $\left(X, d\right) $ is called a $b-$rectangular metric space.
	\end{definition}
	\begin{example}
		\cite{KARO}. Let $ X=A\cup B $, where $ A=\lbrace \frac{1}{n}:n\in\lbrace 2,3,4,5,6,7\rbrace \rbrace $ and $ B=\left[1,2 \right]  $. Define $ d:X\times X\rightarrow \left[0,+\infty \right[  $ as follows:
		\begin{equation*}
			\left\lbrace
			\begin{aligned}
				d(x, y) &=d(y, x)\ for \ all \  x,y\in X;\\
				d(x, y) &=0\Leftrightarrow y= x.	
			\end{aligned}
			\right.
		\end{equation*}
		and
		\begin{equation*}
			\left\lbrace
			\begin{aligned}		    
				d\left( \frac{1}{2},\frac{1}{3}\right) =d\left( \frac{1}{4},\frac{1}{5}\right) =d\left( \frac{1}{6},\frac{1}{7}\right) 	&=0,05\\
				d\left( \frac{1}{2},\frac{1}{4}\right) =d\left( \frac{1}{3},\frac{1}{7}\right) =d\left( \frac{1}{5},\frac{1}{6}\right) 	&=0,08\\
				d\left( \frac{1}{2},\frac{1}{6}\right) =d\left( \frac{1}{3},\frac{1}{4}\right) =d\left( \frac{1}{5},\frac{1}{7}\right) 	&=0,4\\
				d\left( \frac{1}{2},\frac{1}{5}\right) =d\left( \frac{1}{3},\frac{1}{6}\right) =d\left( \frac{1}{4},\frac{1}{7}\right) 	&=0,24\\
				d\left( \frac{1}{2},\frac{1}{7}\right) =d\left( \frac{1}{3},\frac{1}{5}\right) =d\left( \frac{1}{4},\frac{1}{6}\right) 	&=0,15\\
				d\left( x,y\right) =\left( \vert x-y\vert\right) ^{2} \ otherwise.
			\end{aligned}
			\right.
		\end{equation*}
		Then $ (X,d) $ is a b-rectangular metric space with coefficient s=3.
	\end{example}
	\begin{lemma}\label{Lemma 2.3}
		\cite{RO}. Let $\left(X, d\right) $ be a b-rectangular metric space.
		\begin{itemize}
			\item[(a)]  Suppose that sequences $\lbrace  x_{n}\rbrace $ and $\lbrace y_{n}\rbrace $ in $X$ are such that $x_{n}\rightarrow x$ and $y_{n}\rightarrow y$ as $n\rightarrow \infty ,$ with $x\neq y,$ $x_{n}\neq x$ and $y_{n}\neq y$ for all $n\in \mathbb{N}.$ Then we have
			\begin{equation*}
				\frac{1}{s}d\left( x,y\right) \leq \lim_{n\rightarrow \infty }\inf d\left(x_{n},y_{n}\right) \leq \lim_{n\rightarrow \infty }\sup d\left(x_{n},y_{n}\right) \leq sd\left( x,y\right).
			\end{equation*}
			\item[(b)]  if $y\in X$ and $\lbrace x_{n}\rbrace $ is a Cauchy sequence in $X$ with $x_{n}\neq x_{m}$ for any $m,n\in \mathbb{N},$ $m\neq n,$ converging to $x\neq y,$ then
			\begin{equation*}
				\frac{1}{s}d\left( x,y\right) \leq \lim_{n\rightarrow \infty }\inf d\left(x_{n},y\right) \leq \lim_{n\rightarrow \infty }\sup d\left( x_{n},y\right)\leq sd\left( x,y\right),
			\end{equation*}
			for all $x\in X.$
		\end{itemize}
	\end{lemma}
	\begin{lemma}\cite{KAR}\label{Lemma 2.4}
		Let $\left(X, d\right) $ be a b-rectangular metric space and let $ \lbrace x_n \rbrace $ be a sequence in $ X $ such that 
		\begin{equation}
			\lim_{n\rightarrow \infty } d\left(x_{n},x_{n+1}\right)= \lim_{n\rightarrow \infty }\ d\left(x_{n},x_{n+2}\right)=0.
		\end{equation}
		If $ \lbrace x_n \rbrace $ is not a Cauchy sequence, then there exist $ \varepsilon >0 $ and two sequences $ \lbrace m(k) \rbrace $ and $ \lbrace n(k) \rbrace $ of positive integers such that
		$$\varepsilon \leq \lim_{k\rightarrow \infty }\inf d\left( x_{m_{\left( k\right) }},x_{n_{\left( k\right)}}\right)  \leq \lim_{k\rightarrow \infty }\sup d\left( x_{m_{\left( k\right) }},x_{n_{\left( k\right)}}\right)\leq s\varepsilon ,$$ 
		$$\varepsilon \leq \lim_{k\rightarrow \infty }\inf d\left( x_{n_{\left( k\right) }},x_{m_{\left( k\right)+1}}\right)  \leq \lim_{k\rightarrow \infty }\sup d\left( x_{n_{\left( k\right) }},x_{m_{\left( k\right)+1}}\right)\leq s\varepsilon ,$$ 
		$$\varepsilon \leq \lim_{k\rightarrow \infty }\inf d\left( x_{m_{\left( k\right) }},x_{n_{\left( k\right)+1}}\right)  \leq \lim_{k\rightarrow \infty }\sup d\left( x_{m_{\left( k\right) }},x_{n_{\left( k\right)+1}}\right)\leq s\varepsilon ,$$ 
		$$\frac{\varepsilon}{s} \leq \lim_{k\rightarrow \infty }\inf d\left( x_{m_{\left( k\right)+1 }},x_{n_{\left( k\right)+1}}\right)  \leq \lim_{k\rightarrow \infty }\sup d\left( x_{m_{\left( k\right)+1 }},x_{n_{\left( k\right)+1}}\right)\leq s^{2}\varepsilon .$$   
	\end{lemma}
	The following definition introduced by Wardowski. 
	\begin{definition}
		\cite{WAR}. Let $\digamma $ be the family of all functions $F\colon $
		$\mathbb{R}^{+}\rightarrow \mathbb{R}$ such that
		\begin{itemize}
			\item[(i)] $F$ is strictly increasing;
			\item[(ii)] For each sequence $\lbrace x_{n}\rbrace _{n\in \mathbb{N}
			}$ of positive numbers 
			\begin{equation*}
				\lim_{n\rightarrow 0}x_{n}=0,\,\,\,\text{ if and only if }\,\,\,\lim_{n\rightarrow \infty }F\left( x_{n}\right) =-\infty; 
			\end{equation*}
			\item[(iii)] There exists $k\in \left] 0,1\right[ $ such that $\lim_{x\rightarrow 0}x^{k}F\left( x\right) =0.$
		\end{itemize}
	\end{definition}
	Recently, Piri and Kuman \cite{PIRI} extended the result of Wardowski \cite{WAR} by changing the condition $ (iii) $ in Definition $ 2.5 $ as follow. 
	\begin{definition}
		\cite{PIRI}. Let $\Gamma $ be the family of all functions $ F\colon $ $
		\mathbb{R}^{+}\rightarrow \mathbb{R}$ such that
		\begin{itemize}
			\item[(i)] $F$ is strictly increasing;
			\item[(ii)] For each sequence $\lbrace x_{n}\rbrace _{n\in \mathbb{N}
			}$ of positive numbers 
			\begin{equation*}
				\lim_{n\rightarrow\infty}x_{n}=0,\,\,\,\text{ if and only if }%
				\,\,\,\lim_{n\rightarrow \infty }F\left( x_{n}\right) =-\infty;
			\end{equation*}
			\item[(iii)] $F$ is continuous.
		\end{itemize}
	\end{definition}
	The following definition introduced by Wardowski \cite{WAR2} will be used to prove our result.
	\begin{definition}\label{2.7}
		\cite{WAR2}. Let $ \mathbb{F} $ be the family of all functions   $F\colon \mathbb{R}^{+}\rightarrow \mathbb{R}$ and\\  $\phi: \left]0,+\infty \right[\rightarrow \left]0,+\infty \right[$ satisfy the following.
		\begin{itemize} 
			\item[(i)] $F$ is strictly increasing;
			\item[(ii)] For each sequence $\lbrace x_{n}\rbrace _{n\in \mathbb{N}}$ of positive numbers 
			\begin{equation*}
				\lim_{n\rightarrow\infty}x_{n}=0,\,\,\,\text{ if and only if }
				\lim_{n\rightarrow \infty }F\left( x_{n}\right) =-\infty;
			\end{equation*}
			\item[(iii)]  $\liminf_{s \rightarrow \alpha^{+}}\phi(s)> 0 $ for all $ s> 0; $
			\item[(iv)]There exists $ k\in \left]0,1 \right[  $ such that 
			\begin{equation*}
				\lim_{x\rightarrow 0^{+}}x^{k}F(x) =0. 
			\end{equation*}
		\end{itemize}
	\end{definition}
	Inspired by Wardowski \cite{WAR2} and definition introduce by Cosentino et al. in \cite{CJS}, we give the following definition.
	\begin{definition}\label{2.8}
		Let Let $ \mathbb{F} $ be the family of all functions   $F\colon \mathbb{R}^{+}\rightarrow \mathbb{R}$ and $ \Phi $ be the family of all functions $\phi: \left]0,+\infty \right[\rightarrow \left]0,+\infty \right[$ satisfy the following.
		\begin{itemize} 
			\item[(i)] $F$ is strictly increasing;
			\item[(ii)] For each sequence $\lbrace x_{n}\rbrace _{n\in \mathbb{N}}$ of positive numbers 
			\begin{equation*}
				\lim_{n\rightarrow\infty}x_{n}=0,\,\,\,\text{ if and only if }
				\lim_{n\rightarrow \infty }F\left( x_{n}\right) =-\infty;
			\end{equation*}
			\item[(iii)]  $\liminf_{s \rightarrow \alpha^{+}}\phi(s)> 0 $ for all $ s> 0; $
			\item[(iv)]There exists $ k\in \left]0,1 \right[  $ such that 
			\begin{equation*}
				\lim_{x\rightarrow 0^{+}}x^{k}F(x) =0. 
			\end{equation*}
			\item[(v)] ) for each sequence $ \alpha_ {n}\in \mathbb{R^{+}}$ of positive numbers such that $\phi( \alpha_ {n})+F(s \alpha_ {n+1})\leq F( \alpha_ {n})  $ 
			for all $ n\in\mathbb{N} $, then $\phi( \alpha_ {n})+F(s^{n} \alpha_ {n+1})\leq F(s^{n-1} \alpha_ {n})  $  for all $ n\in\mathbb{N} $.
		\end{itemize} 
	\end{definition}
	\begin{example}
		Let $ F: \mathbb{R^{+}}\longrightarrow \mathbb{R^{+}}$ be defined by $F(x) = x + ln x  $. Clearly, $F  $ satisfies $ (i),(ii)\ and \ (iv) $ and $ \phi $ satisfies $ (ii) $. Here we show only $ (v). $
		
		Assume that, for all $ n\in\mathbb{N} $, we have
		$\phi( \alpha_ {n})+ln(s \alpha_ {n+1})\leq ln( \alpha_ {n}). $ 
		Since $ x + ln x $ is an increasing function, then $ s\alpha_ {n+1} < \alpha_ {n} $. Thus
		$$(s^{n-1}-1)s\alpha_ {n+1} +ln((s^{n-1})\leq (s^{n-1}-1)s\alpha_ {n} +ln((s^{n-1})  $$
		implies that
		\begin{align*}
			\phi( \alpha_ {n})+s^{n}\alpha_ {n+1}+ln(s^{n} \alpha_{n+1})&=\phi( \alpha_ {n})+s\alpha_ {n+1}+(s^{n-1}-1)s\alpha_ {n+1}+ln(s^{n-1})+ln(s \alpha_{n+1})\\
			&\leq \alpha_{n}+((s^{n-1}-1)\alpha_{n}+ln(s^{n-1})+ln(\alpha_{n})\\
			&=s^{n-1}\alpha_{n}+ln(s^{n-1}\alpha_{n}).
		\end{align*}
		and hence $ (iv) $ holds true.
	\end{example}
	By replacing the condition $ (iv) $ in definition $ \ref{2.7} $, we introduce a new class of $(\phi, F)-$contraction.
	\begin{definition}\label{2.10}
		Let Let $ \mathbb{\Im} $ be the family of all functions  $F\colon $ $\mathbb{R}^{+}\rightarrow \mathbb{R}$ and $ \Phi $ be the family of all functions $\phi: \left]0,+\infty \right[\rightarrow \left]0,+\infty \right[$ satisfy the following.
		\begin{itemize}  
			\item[(i)] $F$ is strictly increasing;
			\item[(ii)] For each sequence $\lbrace x_{n}\rbrace _{n\in \mathbb{N}}$ of positive numbers 
			\begin{equation*}
				\lim_{n\rightarrow\infty}x_{n}=0,\,\,\,\text{ if and only if }
				\lim_{n\rightarrow \infty }F\left( x_{n}\right) =-\infty;
			\end{equation*}
			\item[(iii)]  $\liminf_{s \rightarrow \alpha^{+}}\phi(s)> 0 $ for all $ s> 0; $
			\item[(iv)] $F$ is continuous.
		\end{itemize}  
	\end{definition}
	\begin{definition}
		\cite{WAR2}. Let $ (X, d) $ be a metric  space. A mapping $ T:X\rightarrow X $ is called an $ \left( \phi,F\right)  $-contraction on $ (X, d) $, if there exist $ F\in\mathbb{F} $  and $ \phi $ such that 
		$$F(d(Tx,Ty))+\phi\left( d(x,y)\right)  \leq F\left( d(x,y)\right)    $$
		for all $ x,y\in X $ for which $ Tx\neq Ty .$
	\end{definition}
	\begin{theorem}
		\cite{WAR2}. Let $\left( X,d\right) $ be a complete  metric space and let $T:X\rightarrow X$ be a $(\phi, F)-$contraction.
		Then $T$ has a unique fixed point. 
	\end{theorem}
	\section{Main result}
	In this paper, using the idea introduced by Wardowski, we present the concept $(\phi, F)-$contraction in  b-rectangular metric spaces and we prove some fixed point results for such spaces.
	\begin{definition}
		Let $(X,d)$ be a b-rectangular metric space with parameter $ s >1 $ space and 
		
		$T:X\rightarrow X$ be a mapping.
		\begin{itemize}  
			\item [1.] $T$ is said to be a $(\phi, F)-$contraction of type $ \left( \mathbb{F}\right)  $ if there exist $ F\in\mathbb{F} $  and $ \phi\in\Phi $ such that 
			\begin{equation*}
				d\left( Tx,Ty\right) >0\Rightarrow F \left[s d\left( Tx,Ty\right)\right]+\phi(d(x,y))  \leq  F \left[ d\left( x,y\right) \right],
			\end{equation*}
			\item [2.] $T$ is said to be a $(\phi, F)-$contraction of type $\left( \mathbb{\Im} \right) $ if there exist $ F\in\mathbb{\Im} $  and $ \phi\in\Phi $ such that 
			\begin{equation*}
				d\left( Tx,Ty\right) >0\Rightarrow F \left[s^{2}d\left( Tx,Ty\right)\right]+\phi(d(x,y))  \leq  F \left[ M\left( x,y\right) \right],
			\end{equation*} 
			where 
			\begin{equation*}
				M\left( x,y\right)=\max \left\{ d\left( x,y\right) ,d\left( x,Tx\right) ,d\left( y,Ty\right),d\left( y,Tx\right)\right\}.
			\end{equation*}
			\item [3.] $T$ is said to be a $(\phi, F)-$Kannan-type $ (\mathbb{\Im} )$ contraction if there exist exist $F \in \mathbb{\Im} $ and $\phi\in\Phi $ such that  $ d\left( Tx,Ty\right)>0, $ 
			we have
			\begin{equation*}
				F\left[ s^{2}d\left( Tx,Ty\right) \right)]+\phi(d(x,y))\leq F\left( \frac{d\left( x,Tx\right)  +d\left( y,Ty\right)  }{2}\right).
			\end{equation*}
			\item [4.] $T$ is said to be a $(\phi, F)-$Reich-type  $ (\mathbb{\Im} )$ contraction if there exist exist $F \in \mathbb{F} $ and $\phi\in\Phi $
			such that $ d\left( Tx,Ty\right)>0, $
			we have 
			\begin{equation*}
				F \left[s^{2} d\left( Tx,Ty\right) \right)]+\phi(d(x,y))\leq  F \left( \frac{d\left( x,y\right)  +d\left( x,Tx\right) +d\left( y,Ty\right) }{3}\right).
			\end{equation*}
		\end{itemize} 
	\end{definition}
	\begin{theorem}
		Let $\left( X,d\right) $ be a complete b-rectangular metric space and let $T:X\rightarrow X$ be an $(\phi, F)-$contraction of type $ \left( \mathbb{F}\right)  $-contraction, i.e, there exist $ F \in \mathbb{F} $ and $\phi $ such that for any $ x,y \in X $, we have
		\begin{equation}\label{3.1}
			d\left( Tx,Ty\right) >0\Rightarrow F \left[sd\left( Tx,Ty\right)\right]+\phi((dx,y))  \leq  F \left[ d\left( x,y\right) \right].
		\end{equation}
		Then $T$ has a unique fixed point.
	\end{theorem}
	\begin{proof}
		Let $x_{0}\in X$ be an arbitrary point in $ X $ and define a sequence $\left\lbrace x_{n}\right\rbrace$ by 
		$$x_{n+1} =Tx_{n}=T^{n+1}x_{0},$$
		for all $n\in \mathbb{N}.$ If there exists $n_0\in \mathbb{N}$ such that $d\left( x_{n_0},x_{n_0+1}\right) =0$, then proof is finished.
		
		We can suppose that $d\left( x_{n},x_{n+1}\right) >0$  for all $n\in \mathbb{N}.$
		
		Substituting $x=x_{n-1}$ and $y=x_{n}$, from $(\ref{3.1})$, for all $n\in  $ $\mathbb{N}$, we have
		\begin{equation}
			F \left[ d\left( x_{n},x_{n+1}\right)\right]\leq F \left[ s d\left( x_{n},x_{n+1}\right)\right]+\phi\left(d(x_{n-1},x_{n}) \right)  \leq  F \left(d\left(x_{n-1},x_{n}\right) \right),\forall n\in \mathbb{N}.
		\end{equation}
		Hence
		\begin{equation}\label{3.3}
			F\left( d\left( x_{n},x_{n+1}\right) \right) < F \left( d\left( x_{n-1},x_{n}\right) \right).
		\end{equation}
		By $(\ref{3.1})  $ and property $ (iv) $ of definition $ \ref{2.8} $, we get
		\begin{equation}
			\phi\left(d(x_{n-1},x_{n} )\right)+F\left( s^{n}d\left( x_{n},x_{n+1}\right) \right) \leq F \left( d\left( s^{n-1}x_{n-1},x_{n}\right) \right)\ for \ all \ n\in\mathbb{N}.
		\end{equation}
		Hence
		\begin{equation}
			F\left( s^{n}d\left( x_{n},x_{n+1}\right) \right) \leq F \left( d\left( s^{n-1}x_{n-1},x_{n}\right) \right)-\phi\left(d(x_{n-1},x_{n}) \right).
		\end{equation}
		Repeating this step, we conclude that 
		\begin{align*}
			F\left( s^{n}d\left( x_{n},x_{n+1}\right) \right)  &\leq F \left(s^{n-1} d\left( x_{n-1},x_{n}\right) \right)-\phi(d\left(x_{n-1},x_{n}\right))\\
			&\leq F \left(s^{n-2} d\left( x_{n-2},x_{n-1}\right) \right)-\phi(d\left(x_{n-1},x_{n}\right)) -\phi(d\left(x_{n-2},x_{n-1}\right)) \\
			&\leq ...\leq F\left( d\left( x_{0},x_{1}\right)\right)-\sum_{i=0}^{n}\phi(d\left(x_{i},x_{i+1}\right)). 
		\end{align*}
		Since $ \liminf _{\alpha \rightarrow s^{+} }\phi(\alpha)> 0 $, we have $ \liminf _{n\rightarrow \infty }\phi( d\left( x_{n-1},x_{n}\right))> 0 $, then from the definition of the limit, there exists $ n_0\in \mathbb{N}$  and $ A> 0 $ such that for all $n\geq n_{0}   $, $ \phi( d\left( x_{n-1},x_{n}\right))>A $. Thus
		\begin{align*}
			F\left(s^{n} d\left( x_{n},x_{n+1}\right) \right)&\leq F\left( d\left( x_{0},x_{1}\right)\right)-\sum_{i=0}^{n_{0}-1}\phi(d\left(x_{i},x_{i+1}\right))-\sum_{i=n_{0}-1}^{n}\phi(d\left(x_{i},x_{i+1}\right))\\
			&\leq F\left( d\left( x_{0},x_{1}\right)\right)-\sum_{i=n_{0}-1}^{n} A\\
			&=F\left( d\left( x_{0},x_{1}\right)\right)-(n-n_0)A,  
		\end{align*}
		for all  $n\geq n_{0} $.
		Taking limit as $ n\rightarrow\infty $ in above inequality we get
		\begin{equation}
			lim _{n \rightarrow\infty }F\left( s^{n}d\left( x_{n},x_{n+1}\right) \right)\leq  \lim _{n \rightarrow\infty }\left[ F\left( d\left( x_{0},x_{1}\right) \right)-(n-n_0)A\right],
		\end{equation}
		that is, $lim _{n \rightarrow\infty }F\left( s^{n}d\left( x_{n},x_{n+1}\right) \right)=-\infty  $ then from the condition $( ii )$ of Definition $ \ref{2.8} $, we conclude that 
		\begin{equation}\label{3.7}
			\lim_{n\rightarrow \infty }s^{n}d\left( x_{n,}x_{n+1}\right)=0.
		\end{equation}
		Next. We shall prove that 
		\begin{equation*}
			\lim_{n\rightarrow \infty }s^{n}d\left( x_{n},x_{n+2}\right) =0.
		\end{equation*}
		We assume that $x_{n}\neq x_{m}$ for every $n,m\in $ $\mathbb{N},\ n\neq m.$ Indeed, suppose that $x_{n}= x_{m}$ for some $  n=m+k$  with $ k>0 $ and using $(\ref{3.3})$, we have  
		\begin{equation}
			d\left( x_{m},x_{m+1}\right) =  d\left( x_{n},x_{n+1}\right) <d\left( x_{n-1},x_{n}\right).
		\end{equation}
		Continuing this process, we can that
		\begin{equation*} 
			d\left( x_{m},x_{n+1}\right)= d\left( x_{n},x_{n+1}\right)<d\left( x_{m},x_{m+1}\right).
		\end{equation*}
		It is a contradiction. Therefore, $d\left( x_{n},x_{m}\right)>0  $ for every $n,m\in $ $\mathbb{N}$, $n\neq m .$
		
		By $(\ref{3.1})  $ and property $ (iv) $ of Definition $ \ref{2.8} $, we get
		\begin{equation}
			\phi\left(d(x_{n-1},x_{n+1}) \right)+F\left( s^{n}d\left( x_{n},x_{n+2}\right) \right) \leq F \left( d\left( s^{n-1}x_{n-1},x_{n+1}\right) \right)\ for \ all \ n\in\mathbb{N}.
		\end{equation}
		Hence
		\begin{equation}
			F\left( s^{n}d\left( x_{n},x_{n+2}\right) \right) \leq F \left( d\left( s^{n-1}x_{n-1},x_{n+1}\right) \right)-\phi\left(x_{n-1},x_{n+1} \right).
		\end{equation}
		Repeating this step, we conclude that 
		\begin{align*}
			F\left( s^{n}d\left( x_{n},x_{n+2}\right) \right)  &\leq F \left(s^{n-1} d\left( x_{n-1},x_{n+1}\right) \right)-\phi(d\left(x_{n-1},x_{n+1}\right))\\
			&\leq F \left(s^{n-2} d\left( x_{n-2},x_{n}\right) \right)-\phi(d\left(x_{n-1},x_{n+1}\right)) -\phi(d\left(x_{n-2},x_{n}\right)) \\
			&\leq ...\leq F\left( d\left( x_{0},x_{2}\right)\right)-\sum_{i=0}^{n}\phi(d\left(x_{i},x_{i+2}\right)). 
		\end{align*}
		Since $ \liminf _{\alpha \rightarrow s^{+} }\phi(\alpha)> 0 $, we have $ \liminf _{n\rightarrow \infty }\phi( d\left( x_{n-1},x_{n+1}\right))> 0 $, then from the definition of the limit, there exists $ n_1\in \mathbb{N}$  and $ B> 0 $ such that for all $n\geq n_{1}   $, $ \phi( d\left( x_{n-1},x_{n+1}\right))>B $. Thus
		\begin{align*}
			F\left(s^{n} d\left( x_{n},x_{n+2}\right) \right)&\leq F\left( d\left( x_{0},x_{2}\right)\right)-\sum_{i=0}^{n_{1}-1}\phi(d\left(x_{i},x_{i+2}\right))-\sum_{i=n_{1}-1}^{n}\phi(d\left(x_{i},x_{i+2}\right))\\
			&\leq F\left( d\left( x_{0},x_{2}\right)\right)-\sum_{i=n_{1}-1}^{n} B\\
			&=F\left( d\left( x_{0},x_{2}\right)\right)-(n-n_1)B,  
		\end{align*}
		for all  $n\geq n_{1} $.
		Taking limit as $ n\rightarrow\infty $ in above inequality we get
		\begin{equation}
			lim _{n \rightarrow\infty }F\left( s^{n}d\left( x_{n},x_{n+2}\right) \right)\leq  \lim _{n \rightarrow\infty }\left[ F\left( d\left( x_{0},x_{2}\right) \right)-(n-n_0)B\right],
		\end{equation}
		that is, $lim _{n \rightarrow\infty }F\left( s^{n}d\left( x_{n},x_{n+2}\right) \right)=-\infty  $ then from the condition $( ii )$ of Definition $ (\ref{2.8}) $, we conclude that 
		\begin{equation}\label{3.12}
			\lim_{n\rightarrow \infty }s^{n}d\left( x_{n,}x_{n+2}\right)=0.
		\end{equation}
		Next, We shall prove that $\left\lbrace  x_{n}\right\rbrace  _{n\in \mathbb{N}}$ is a Cauchy sequence, i.e, $\lim_{n,m\rightarrow \infty }d\left( x_{n,}x_{m}\right) =0,$ for all $n,m\in \mathbb{N}$.
		Now, from $\left(iv\right)$ of Definition $ \ref{2.8} $, there exists $k\in \left] 0,1\right[ $ such that
		\begin{equation}
			\lim_{n\rightarrow \infty }\left[s^{n} d\left(x_{n},x_{n+1}\right) \right]^{k}F\left(s^{n} d\left(x_{n},x_{n+1}\right) \right) =0.
		\end{equation} 
		Since
		\begin{equation*}
			F\left[ s^{n}d\left(x_{n},x_{n+1}\right)\right] \leq F\left[ d\left(x_{0},x_{1}\right) \right] -(n-n_0)A,
		\end{equation*}
		we have
		\begin{align*}
			\left[ s^{n}d\left(x_{n},x_{n+1}\right) \right] ^{k}F\left[ s^{n} d\left(x_{n},x_{n+1}\right) \right]
			&\leq \left[ s^{n}d\left( x_{n},x_{n+1}\right)\right] ^{k}F\left[ d\left( x_{0},x_{1}\right) \right] -\left[ (n-n_0)A\right]  \left[ s^{n}d\left( x_{n},x_{n+1}\right) \right] ^{k}
		\end{align*}
		Therefore,
		\begin{align*}
			\left[s^{n}d\left( x_{n},x_{n+1}\right) \right] ^{k}F\left[ s^{n}d\left(x_{n},x_{n+1}\right)  \right]
			&-\left[ s^{n}d\left( x_{n},x_{n+1}\right)  \right] ^{k}F\left[  d\left(x_{0},x_{1}\right) \right]\\
			&\leq -\left[ (n-n_0)A\right] \left[ s^{n}d\left( x_{n},x_{n+1}\right) \right] ^{k}
			\leq 0.
		\end{align*}
		Taking limit $ n\rightarrow\infty $ in above inequality, we conclude that
		$$lim_{n\rightarrow \infty }s^{n}d\left( x_{n,}x_{n+1}\right)^{k}(n-n_0)A=0.  $$
		Then there exists $ h\in\mathbb{N} $, such that ,
		\begin{equation}\label{3.14}
			s^{n}d\left(x_{n,}x_{n+1}\right)\leq \frac{1}{\left[ (n-n_0)A\right] ^{k}}\ for\  all \ n\geq h. 
		\end{equation} 
		Now, from $\left(iv\right)$ of Definition $ \ref{2.8} $, there exists $k\in \left] 0,1\right[ $ such that
		\begin{equation}
			\lim_{n\rightarrow \infty }\left[s^{n} d\left(x_{n},x_{n+2}\right) \right]^{k}F\left(s^{n} d\left(x_{n},x_{n+2}\right) \right) =0.
		\end{equation} 
		Since
		\begin{equation*}
			F\left[ s^{n}d\left(x_{n},x_{n+2}\right)\right] \leq F\left[ d\left(x_{0},x_{2}\right) \right] -(n-n_1)B,
		\end{equation*}
		we have
		\begin{align*}
			\left[ s^{n}d\left(x_{n},x_{n+2}\right) \right] ^{k}F\left[ s^{n} d\left(x_{n},x_{n+2}\right) \right]
			&\leq \left[ s^{n}d\left( x_{n},x_{n+2}\right)\right] ^{k}F\left[ d\left( x_{0},x_{2}\right) \right] -\left[ (n-n_1)B\right]  \left[ s^{n}d\left( x_{n},x_{n+2}\right) \right] ^{k}
		\end{align*}
		Therefore,
		\begin{align*}
			\left[s^{n}d\left( x_{n},x_{n+2}\right) \right] ^{k}F\left[ s^{n}d\left(x_{n},x_{n+2}\right)  \right]
			&-\left[ s^{n}d\left( x_{n},x_{n+2}\right)  \right] ^{k}F\left[  d\left(x_{0},x_{2}\right) \right]\\
			&\leq -\left[ (n-n_1)B\right] \left[ s^{n}d\left( x_{n},x_{n+2}\right) \right] ^{k}
			\leq 0.
		\end{align*}
		Taking limit $ n\rightarrow\infty $ in above inequality, we conclude that
		$$lim_{n\rightarrow \infty }s^{n}d\left( x_{n,}x_{n+2}\right)^{k}(n-n_1)B=0  $$
		Then there exists $ l\in\mathbb{N} $, such that,
		\begin{equation}\label{3.16}
			s^{n}d\left(x_{n,}x_{n+2}\right)\leq \frac{1}{\left[ (n-n_1)B\right] ^{k}}\  for\ all \ n\geq l. 
		\end{equation}
		Next, we show that  $\lbrace x_{n}\rbrace _{n\in \mathbb{N}}$ is a Cauchy sequence, i.e,
		\begin{equation*}
			\lim_{n\rightarrow \infty }d\left( x_{n},x_{n+r}\right) =0 \ for \ all \ r\in \mathbb{N^{*}}.
		\end{equation*}
		The cases $r=1$ and $\ r=2,$ are proved, respectively by $(\ref{3.7})
		$ and $(\ref{3.12}). $
		
		Now, we take $r\geq 3$. It is sufficient to examine two cases:\\
		Case $(I)$: Suppose that $\ r=2m+1$ where $m\geq 1$.By using the quadrilateral inequality together we have
		\begin{align*}
			d\left( x_{n},x_{n+r}\right) &=d\left( x_{n},x_{n+2m+1}\right)\\
			& \leq s\left[ d\left(x_{n},x_{n+1}\right) +d\left( x_{n+1},x_{n+2}\right) +\left( x_{n+2},x_{n+2m+1}\right)\right]\\
			&\leq s\left[ d\left(x_{n},x_{n+1}\right) +d\left( x_{n+1},x_{n+2}\right)\right]  +s^{2}\left[ \left( x_{n+2},x_{n+3}\right) +\left( x_{n+3},x_{n+4}\right)\right] \\
			&+...+s^{m} d\left( x_{n+2m},x_{n+2m+1}\right)\\ 
			&=\frac{1}{s^{n-1}}\left[  s^{n}d\left(x_{n},x_{n+1}\right) +s^{n} d\left( x_{n+1},x_{n+2}\right)\right]+\frac{1}{s^{n-1}}\left[s^{n+1}d\left(x_{n+2},x_{n+3}\right) +s^{n+1} d\left( x_{n+3},x_{n+4}\right)  \right]\\
			&+...+\frac{1}{s^{n-1}}\left[s^{m+n-1}d\left(x_{n+2m},x_{n+2m-1}\right)  \right]\\
			&=\frac{1}{s^{n-1}}\left[ s^{n}d\left(x_{n},x_{n+1}\right) +s^{n+1}d\left(x_{n+2},x_{n+3}\right)+... +s^{m} d\left( x_{n+2m-2},x_{n+n+2m-1}\right)\right] \\
			&+ \frac{1}{s^{n-1}}\left[ s^{n}d\left(x_{n+1},x_{n+2}\right) +s^{n+1}d\left(x_{n+3},x_{n+4}\right)+...+ s^{m} d\left( x_{n+2m-1},x_{n+n+2m}\right)\right]\\
			&+\frac{1}{s^{n-1}}\left[  s^{m}d\left(x_{n+2m},x_{n+2m+1}\right)\right]\\ 
			&\leq \frac{1}{s^{n-1}}\left[ s^{n}d\left(x_{n},x_{n+1}\right) +s^{n+2}d\left(x_{n+2},x_{n+3}\right)+... +s^{n+2m-2} d\left( x_{n+2m-2},x_{n+2m-1}\right)\right] \\
			&+ \frac{1}{s^{n-1}}\left[ s^{n+1}d\left(x_{n+1},x_{n+2}\right) +s^{n+3}d\left(x_{n+3},x_{n+4}\right)+...+ s^{n+2m-1} d\left( x_{n+2m-1},x_{n+n+2m}\right)\right]\\
			&+s^{m+2m} \left[ d\left( x_{n+2m},x_{n+n+2m-1}\right)\right] \\
			&= \frac{1}{s^{n-1}}\sum_{i=n}^{i=n+2m}s^{i}d\left( x_{i},x_{i+1}\right)\\
			&= \frac{1}{s^{n-1}}\sum_{i=n}^{i=n+r-1}s^{i}d\left( x_{i},x_{i+1}\right).
		\end{align*}
		Hence, for all $n\geq  \max\lbrace n_{0},n_{h} \rbrace  $ and $ r\in\mathbb{N^{*}}$ inequality $ (\ref{3.14}) $ implies
		\begin{equation*}
		d\left( x_{n},x_{n+r}\right)\leq \frac{1}{s^{n-1}}\sum_{i=n}^{i=n+r-1}s^{i}d\left( x_{i},x_{i+1}\right)\leq \frac{1}{s^{n-1}}\sum_{i=n}^{i=\infty}s^{i}d\left( x_{i},x_{i+1}\right)\leq  \frac{1}{s^{n-1}}\sum_{i=n}^{i=\infty}\frac{1}{\left[ (i-n_0)A\right] ^{k}}\rightarrow 0.  
		\end{equation*}
		Case $(II)$: Suppose that $\ r=2m$ where $m\geq 1$. By using the
		quadrilateral inequality together we have
		\begin{align*}
			d\left( x_{n},x_{n+r}\right) &=d\left( x_{n},x_{n+2m}\right)\\
			& \leq s\left[ d\left(x_{n},x_{n+2}\right) +d\left( x_{n+2},x_{n+3}\right) +\left( x_{n+3},x_{n+2m}\right)\right]\\
			&\leq s\left[ d\left(x_{n},x_{n+2}\right) +d\left( x_{n+2},x_{n+3}\right)\right]  +s^{2}\left[ \left( x_{n+3},x_{n+4}\right) +\left( x_{n+4},x_{n+5}\right)\right] \\
			&+...+s^{m-1} d\left( x_{n+2m-3},x_{n+n+2m-2}\right)+s^{m-1} d\left( x_{n+2m-2},x_{n+n+2m-1}\right)+s^{m-1} d\left( x_{n+2m-1},x_{n+n+2m}\right)\\ 
			&=\frac{1}{s^{n-1}}\left[  s^{n}d\left(x_{n},x_{n+2}\right) +s^{n} d\left( x_{n+2},x_{n+3}\right)\right]+\frac{1}{s^{n-1}}\left[s^{n+1}d\left(x_{n+3},x_{n+4}\right) +s^{n+1} d\left( x_{n+4},x_{n+5}\right)  \right]\\
			&+...+\frac{1}{s^{n-1}}\left[s^{m+n-2}d\left(x_{n+2m-3},x_{n+2m-2}\right)  \right]+\frac{1}{s^{n-1}}\left[s^{m+n-2}d\left(x_{n+2m-2},x_{n+2m-1}\right)\right]\\
			&+\frac{1}{s^{n-1}}\left[s^{m+n-2}d\left(x_{n+2m-1},x_{n+2m}\right)  \right]\\
			&=\frac{1}{s^{n-1}}\left[  s^{n}d\left(x_{n},x_{n+2}\right) +s^{n} d\left( x_{n+2},x_{n+3}\right)\right]\\
			&+\frac{1}{s^{n-1}}\left[ s^{n+1}d\left(x_{n+3},x_{n+4}\right) +s^{n+2}d\left(x_{n+5},x_{n+6}\right)+... +s^{m-1} d\left( x_{n+2m-3},x_{n+2m-2}\right)\right] \\
			&+ \frac{1}{s^{n-1}}\left[ s^{n+1}d\left(x_{n+4},x_{n+5}\right) +s^{n+2}d\left(x_{n+6},x_{n+7}\right)+...+ s^{m-2} d\left( x_{n+2m-2},x_{n+n+2m-1}\right)\right]\\
			&+\frac{1}{s^{n-1}}\left[ s^{m-1}d\left(x_{n+2m-1},x_{n+2m}\right)\right] \\
			&\leq\frac{1}{s^{n-1}}  s^{n}d\left(x_{n},x_{n+2}\right) +\frac{1}{s^{n-1}}s^{n+2} d\left( x_{n+2},x_{n+3}\right)\\
			&+\frac{1}{s^{n-1}}\left[ s^{n+3}d\left(x_{n+3},x_{n+4}\right) +s^{n+5}d\left(x_{n+5},x_{n+6}\right)+... +s^{n+2m-3} d\left( x_{n+2m-3},x_{n+2m-2}\right)\right] \\
			&+ \frac{1}{s^{n-1}}\left[ s^{n+4}d\left(x_{n+4},x_{n+5}\right) +s^{n+6}d\left(x_{n+6},x_{n+7}\right)+...+ s^{n+2m-2} d\left( x_{n+2m-2},x_{n+n+2m-1}\right)\right]\\
			&+\frac{1}{s^{n-1}}\left[ s^{n+2m-1}d\left(x_{n+2m-1},x_{n+2m}\right)\right] \\
			&= \frac{1}{s^{n-1}}  s^{n}d\left(x_{n},x_{n+2}\right)+ \frac{1}{s^{n-1}}\sum_{i=n+2}^{i=n+2m-1}s^{i}d\left( x_{i},x_{i+1}\right)\\
			&= \frac{1}{s^{n-1}}  s^{n}d\left(x_{n},x_{n+2}\right)+\frac{1}{s^{n-1}}\sum_{i=n+2}^{i=n+r-1}s^{i}d\left( x_{i},x_{i+1}\right).
		\end{align*}
		Hence, for all $n\geq  \max\lbrace n_{0},n_{0},n_{1},n_{l} \rbrace  $ and $ r\in\mathbb{N^{*}}$ inequality $ (\ref{3.14}) $ and  $ ( \ref{3.16}) $ implies
		\begin{align*}
			\frac{1}{s^{n-1}}s^{n}d\left(x_{n},x_{n+2}\right)+ \frac{1}{s^{n-1}}\sum_{i=n+2}^{i=n+r-1}s^{i}d\left( x_{i},x_{i+1}\right)&\leq \frac{1}{s^{n-1}}s^{n}d\left(x_{n},x_{n+2}\right)+ \frac{1}{s^{n-1}}\sum_{i=n+2}^{i=\infty}s^{i}d\left( x_{i},x_{i+1}\right)\\
			&\leq \left[ \frac{1}{s^{n-1}}\frac{1}{\left[ (n-n_1)B\right] ^{k}}+ \frac{1}{s^{n-1}}\sum_{i=n+2}^{i=\infty}\frac{1}{\left[ (i-n_0)A\right] ^{k}}\right] \rightarrow 0. 
		\end{align*}
		Thus
		\begin{equation*}
			\lim_{n\rightarrow \infty }d\left( x_{n},x_{n+r}\right) =0.
		\end{equation*}
		Consequently, $\lbrace x_{n}\rbrace _{n\in \mathbb{N}}$ is a Cauchy sequence in $ X $. By completeness of $\left( X,d\right) ,$ there exists $z\in X$ such that 
		\begin{equation*}
			\lim_{n\rightarrow \infty }d\left( x_{n},z\right)  =0.
		\end{equation*}
		Now, we show that $d\left( Tz,z\right) =0$  arguing by contradiction, we assume that
		\begin{equation*}
			d\left( Tz,z\right)>0.
		\end{equation*}
		On the other hand, since $ \phi(x,y) + F(d(Tx, Ty)) \leq \phi(d(x,y)) + F(d(Tx, Ty))\leq F(d(x, y))  $
		holds for all such $x, y \in X  $ for which $ d(Tx, Ty) > 0 $, and because F is increasing,
		$ d(Tx, Ty)\leq d(x, y)  $ for all $ x,y\in X .$
		This implies
		$$d\left( Tx_{n},Tz\right) \leq d\left( x_{n},z\right) \ for \ all\  n\in\mathbb{N}.$$  
		Since $ x_{n}\rightarrow z $ as $ n\rightarrow \infty $ for all $ n\in \mathbf{N} $, then from Lemma $ \ref{Lemma 2.3} $, we conclude that
		\begin{equation*}
			\frac{1}{s}d\left( z,Tz\right)\leq \lim_{n\rightarrow \infty }\sup d\left(Tx_{n},Tz\right) \leq sd\left( z,Tz\right).
		\end{equation*}
		Hence
		\begin{equation*}
			\frac{1}{s}d\left( z,Tz\right)\leq \lim_{n\rightarrow \infty }\sup d\left(Tx_{n},Tz\right) \leq\lim_{n\rightarrow \infty }\sup d\left(x_{n},z\right)=0. 
		\end{equation*}
		Hence $Tz=z$.
		
		Uniqueness. Now, suppose that $z,u\in X$ are two fixed points of $T$ such that $u\neq z$. Therefore, we have
		\begin{equation*}
			d\left( z,u\right) =d\left( Tz,Tu\right)  >0.
		\end{equation*}
		Applying $( \ref{3.1}) $ with $ x=z $ and $ y=u $, we have
		\begin{align*}
			F \left( d\left( z,u\right)\right)&=F \left( d\left( Tu,Tz\right)\right)\\
			&\leq F \left(s d\left( Tu,Tz\right)\right)\\
			& \leq F\left( d\left( z,u\right)\right)-\phi\left(d(z,u )\right)\\
			&<F\left( d\left( z,u\right)\right). 
		\end{align*}
		Which implies that
		\begin{equation*}
			d\left( z,u\right) <d\left( z,u\right).
		\end{equation*}
		It is a contradiction. Therefore $u=z$.
	\end{proof}
	\begin{corollary}
		Let $d\left( X,d\right) $ be a complete b-rectangular metric space with parameter $s>1$ and let $T$ be a self mapping on $\ X$. If for all $x,y\in X$ we have 
		\begin{equation*}
			d\left( Tx,Ty\right) >0\Rightarrow  sd\left( Tx,Ty\right) \leq  e^{\frac{-1}{1+d\left( x,y\right)} }d\left( x,y\right) 
		\end{equation*}
		Then $T$ has a unique fixed point.
	\end{corollary}
	\begin{proof}
		Since $ d\left( Tx,Ty\right) >0 $ then we can take natural logarithm sides and get
		\begin{align*}
			ln(sd\left( Tx,Ty\right))& \leq  ln\left[ e^{\frac{-1}{1+d\left( x,y\right)} }d\left( x,y\right)\right] \\
			&= \frac{-1}{1+d\left( x,y\right)}+ ln\left[ d\left( x,y\right)\right]. 
		\end{align*}
		Hence
		\begin{equation*}
			F\left[ sd\left( Tx,Ty\right)\right] +\phi\left(d( x,y)\right)  \leq F\left( d\left( x,y\right)\right) 
		\end{equation*}
		with $ F(t)=ln(t) $ and $ \phi(t)=\frac{1}{1+t} .$ 
	\end{proof}
	\begin{example}
		Let $ X=A\cup B $, where $ A=\lbrace 0,\frac{1}{2},\frac{1}{3},\frac{1}{4}\rbrace $ and $ B=\left[1,2 \right]  $.
		
		Define $ d:X\times X\rightarrow \left[0,+\infty \right[  $ as follows:
		\begin{equation*}
			\left\lbrace
			\begin{aligned}
				d(x, y) &=d(y, x)\ for \ all \  x,y\in X;\\
				d(x, y) &=0\Leftrightarrow y= x.	
			\end{aligned}
			\right.
		\end{equation*}
		and
		\begin{equation*}
			\left\lbrace
			\begin{aligned}		    
				d\left( 0,\frac{1}{2}\right) =d\left( \frac{1}{2},\frac{1}{3}\right)=0,16\\
				d\left(0,\frac{1}{3}\right) =d\left( \frac{1}{3},\frac{1}{4}\right)	=0,04\\
				d\left(0,\frac{1}{4}\right) =d\left( \frac{1}{2},\frac{1}{4}\right)	=0,25\\
				d\left( x,y\right) =\left( \vert x-y\vert\right) ^{2} \ otherwise.
			\end{aligned}
			\right.
		\end{equation*}
		Then $ (X,d) $ is a b-rectangular metric space with coefficient s=3. However we have the following:
		\item[1)] $ (X,d) $ is not a  metric space, as $d\left( 0,\frac{1}{4}\right)=0.25>0.08=d\left(0,\frac{1}{3}\right)+d\left( \frac{1}{3},\frac{1}{4}\right)$.  
		\item[2)] $ (X,d) $ is not a rectangular metric space, as $d\left( \frac{1}{2},\frac{1}{4}\right)=0.25>0.24=d\left( \frac{1}{2},0\right)+d\left( 0,\frac{1}{3}\right)+d\left( \frac{1}{3},\frac{1}{4}\right)$.
		
		Define mapping $T:X\rightarrow X$ by
		\begin{equation*}
			T(x)=\left\lbrace
			\begin{aligned}
				x^{\frac{1}{4}}	& \ if \ x\in \left[1,2 \right]\\
				1&  \ if \ x\in A.
			\end{aligned}
			\right.
		\end{equation*}
		Evidently, $ T(x)\in X $. Let $F( t) =ln(t)+\sqrt{t},$ $\phi (t)=\frac{1}{1+t}$. It obvious that $F \in \mathbb{F}$ and $\phi\in\Phi .$
		
		Consider the following possibilities:
		
		case $ 1 $ : $ x,y \in  \left[1,2 \right] \  x\neq y$. Then
		$$ T(x)=x^{\frac{1}{4}}	,\   T(y)=y^{\frac{1}{4}}	,\ d\left( Tx,Ty \right)=\left( x^{\frac{1}{4}}-y^{\frac{1}{4}}\right) ^{2},\ d(x,y)=(x-y)^{2}.  $$
		On the other hand 
		$$ F \left[sd\left(Tx,Ty\right) \right]=ln(3(x^{\frac{1}{4}}-y^{\frac{1}{4}})^{2})+ \sqrt{3}(x^{\frac{1}{4}}-y^{\frac{1}{4}}),   $$
		$$ F \left[d\left(x,y\right)\right]=ln((x-y)^{2})+ (x-y) $$
		and
		$$ \phi \left[d\left(x,y\right)\right]=\frac{1}{\left[ 1+(x-y)^{2} \right]}  .$$
		We have 
		\begin{align*}
			F \left[sd\left(Tx,Ty\right) \right]+\phi \left[d\left(x,y\right)\right]-F \left[d\left(x,y\right)\right]&=ln(3(x^{\frac{1}{4}}-y^{\frac{1}{4}})^{2})-ln((x-y)^{2})+\sqrt{3}(x^{\frac{1}{4}}-y^{\frac{1}{4}}) -(x-y)+ \frac{1}{\left[ 1+(x-y)^{2} \right]}\\
			&=ln(3(x^{\frac{1}{4}}-y^{\frac{1}{4}})^{2})-ln((x-y)^{2})+ \sqrt{3}(x^{\frac{1}{4}}-y^{\frac{1}{4}})- (x-y)+\frac{1}{\left[ 1+(x-y)^{2} \right]} \\
			&=-2ln\left( \frac{x^{\frac{1}{4}} +y^{\frac{1}{4}}}{\sqrt{3}}\right)-2ln\left(x^{\frac{1}{2}}+y^{\frac{1}{2}} \right)\\
			& +(x^{\frac{1}{4}}-y^{\frac{1}{4}})\left( \sqrt{3}-\left( x^{\frac{1}{4}}+y^{\frac{1}{4}}\right)\left( x^{\frac{1}{2}}+y^{\frac{1}{2}}\right) \right) +\frac{1}{\left[ 1+(x-y)^{2} \right]}. 
		\end{align*}
		Since  $ x,y \in  \left[1,2 \right]  $, then
		$$  \frac{ x^{\frac{1}{4}}+ y^{\frac{1}{4}}}{\sqrt{3}}\geq 1\Rightarrow -2ln\left( \frac{ x^{\frac{1}{4}}+ x^{\frac{1}{4}}}{\sqrt{3}}\right)\leq 0  ,$$
		$$  (x^{\frac{1}{4}}-y^{\frac{1}{4}})\left( \sqrt{3}-\left( x^{\frac{1}{4}}+y^{\frac{1}{4}}\right)\left( x^{\frac{1}{2}}+y^{\frac{1}{2}}\right) \right)\leq0  $$
		and
		$$-2ln\left(x^{\frac{1}{2}}+y^{\frac{1}{2}} \right) +\frac{1}{\left[ 1+\left(  x^{\frac{1}{4}}-y^{\frac{1}{4}}\right) ^{2}\right] }\leq -2ln\left(x^{\frac{1}{2}}+y^{\frac{1}{2}} \right) +1\leq0.     $$ 
		Thus, for all $ x,y\in \left[1,2 \right] $ with $ x\neq y $, we have
		$$  F \left[sd\left(Tx,Ty\right) \right]+\phi \left[d\left(x,y\right)\right]\leq F \left[d\left(x,y\right)\right] $$
		case $ 2 $ : $ x \in  \left[1,2 \right]  $, $ y\in A $. Then
		$$ T(x)=x^{\frac{1}{4}}	,\   T(y)=1	,\ d\left( Tx,Ty \right)=\left( x^{\frac{1}{4}}-1\right) ^{2},\ d(x,y)=(x-y)^{2}.  $$
		On the other hand 
		$$ F \left[sd\left(Tx,Ty\right) \right]=ln(3(x^{\frac{1}{4}}-1)^{2})+ \sqrt{3}(x^{\frac{1}{4}}-1),   $$
		$$ F \left[d\left(x,y\right)\right]=ln((x-y)^{2})+ (x-y) $$
		and
		$$ \phi \left[d\left(x,y\right)\right]=\frac{1}{\left[ 1+(x-y)^{2}\right] }.$$
		We have 
		\begin{align*}
			F \left[d\left(x,y\right)\right]- F \left[sd\left(Tx,Ty\right) \right]-\phi \left[d\left(x,y\right)\right]&=  (x-y) -\sqrt{3}(x^{\frac{1}{4}}-1)+ln((x-y)^{2})-ln\left[ 3(x^{\frac{1}{4}}-1)^{2}\right]\\
			& -\frac{1}{\left[ 1+(x-y)^{2}\right] }\\
			&=2ln\left[ \frac{x-y }{\sqrt{3}(x^{\frac{1}{4}}-1)}\right] +(x-y)-\sqrt{3}(x^{\frac{1}{4}}-1) -\frac{1}{\left[ 1+(x-y)^{2}\right] }. 
		\end{align*}
		Since  $ x \in  \left[1,2 \right]  $ and $ y\in A $, then
		$$(x-y)^{2}\geq \left( x-\frac{1}{2}\right) ^{2}=\left( x-1+\frac{1}{2}\right) ^{2}>\left( x-1\right) ^{2}  .$$
		Hence 
		$$(x-y)>\left( x-1\right)=\left( x^{\frac{1}{4}}-1 \right)\left( x^{\frac{1}{4}}+1 \right) \left( x^{\frac{1}{2}}+1 \right),   $$
		$$(x-y) -\sqrt{3}\left( x^{\frac{1}{4}}-1 \right)> \left( x^{\frac{1}{4}}-1 \right)\left[ \left( x^{\frac{1}{4}}+1 \right) \left( x^{\frac{1}{2}}+1 \right))-\sqrt{3}\right]    $$
		and
		$$\frac{(x-y)}{\sqrt{3}\left( x^{\frac{1}{4}}-1 \right)}>\frac{\left( x^{\frac{1}{4}}+1 \right) \left( x^{\frac{1}{2}}+1 \right)}{\sqrt{3}}. $$
		Then we have 
		$$2ln\left[ \frac{x-y }{\sqrt{3}(x^{\frac{1}{4}}-1)}\right]> 2 ln\left[ \frac{\left( x^{\frac{1}{4}}+1 \right) \left( x^{\frac{1}{2}}+1 \right) }{\sqrt{3}}\right] =2 ln\left[ \frac{\left( x^{\frac{1}{4}}+1 \right)  }{\sqrt{3}}\right]+2 ln\left[  \left( x^{\frac{1}{2}}+1 \right)\right].$$
		Since $ x\in\left[1.2 \right]  $, then 
		$$ 2 ln\left[ \frac{\left( x^{\frac{1}{4}}+1 \right)  }{\sqrt{3}}\right]\geq 0 \ and \ 2 ln\left[  \left( x^{\frac{1}{2}}+1 \right)\right]\geq 1.$$
		On the other hand 
		$$\frac{1}{\left[ 1+(x-y)^{2}\right] }\leq 1 .$$
		Thus, for all $ x\in \left[1,2 \right] $ and $ y\in A  $, we have
		$$  F \left[sd\left(Tx,Ty\right) \right]+\phi \left[d\left(x,y\right)\right]\leq F \left[d\left(x,y\right)\right] $$
		Hence, the condition $ (\ref{3.1}) $ is satisfied. Therefore, $ T $ has a unique fixed point $ z=1 $.  
	\end{example}
	\begin{theorem}\label{3.5}
		Let $\left( X,d\right) $ be a complete b-rectangular metric space and let $T:X\rightarrow X$ be an $F -\phi-$contraction of type $ \left( \Im\right)  $-contraction, i.e, there exist $ F \in\Im $ and $\phi $ such that for any $ x,y \in X $, we have
		\begin{equation}\label{3.17}
			d\left( Tx,Ty\right) >0\Rightarrow F \left[s^{2}d\left( Tx,Ty\right)\right]+\phi(d(x,y))  \leq  F \left[ M\left( x,y\right) \right].
		\end{equation}
		Then $T$ has a unique fixed point.
	\end{theorem}
	\begin{proof}
		Let $x_{0}\in X$ be an arbitrary point in $ X $ and define a sequence $\left\lbrace x_{n}\right\rbrace$ by 
		$$x_{n+1} =Tx_{n}=T^{n+1}x_{0},$$
		for all $n\in \mathbb{N}.$ If there exists $n_0\in \mathbb{N}$ such that $d\left( x_{n_0},x_{n_0+1}\right) =0$, then proof is finished.
		
		We can suppose that $d\left( x_{n},x_{n+1}\right) >0$  for all $n\in \mathbb{N}.$
		
		Substituting $x=x_{n-1}$ and $y=x_{n}$, from $(\ref{3.17})$, for all $n\in  $ $\mathbb{N}$, we have
		\begin{equation}\label{3.18}
			F \left[ d\left( x_{n},x_{n+1}\right)\right]\leq F \left[ s^{2}d\left( x_{n},x_{n+1}\right)\right]+\phi\left(d(x_{n-1},x_{n}) \right)  \leq  F \left( M\left(x_{n-1},x_{n}\right) \right) ,\forall n\in \mathbb{N}
		\end{equation}
		where 
		\begin{align*}
			M\left( x_{n-1},x_{n}\right) &=\max\left(d\left(x_{n-1},x_{n}\right),d\left(x_{n-1},x_{n}\right),d\left(x_{n},x_{n+1}\right) ,d\left(x_{n+1},x_{n+1}\right)\right\}\\
			&=\max\lbrace d\left(x_{n-1},x_{n}\right),d\left(x_{n},x_{n+1}\right)\rbrace. 
		\end{align*} 
		If $M\left( x_{n-1},x_{n}\right) =d\left(x_{n},x_{n+1}\right),$ by $ (\ref{3.18}) $, we have
		\begin{equation*}
			F \left( d\left(x_{n},x_{n+1}\right)\right) \leq F \left(d\left(x_{n},x_{n+1}\right)\right)-\phi\left(d(x_{n-1},x_{n}) \right)<F \left(d\left(x_{n},x_{n+1}\right)\right).
		\end{equation*}
		Since $ F $ is increasing, we have 
		\begin{equation}\label{3.19}
			d\left( x_{n},x_{n+1}\right) <  d\left( x_{n-1},x_{n}\right).
		\end{equation}
		It is a contradiction. Hence, $M\left( x_{n-1},x_{n}\right) =d\left(x_{n-1},x_{n}\right).$ Thus,
		\begin{equation}\label{3.20}
			F\left( d\left( x_{n},x_{n+1}\right) \right) \leq F \left( d\left( x_{n-1},x_{n}\right) \right)-\phi(d\left(x_{n-1},x_{n}\right)).
		\end{equation}
		Repeating this step, we conclude that 
		\begin{align*}
			F\left( d\left( x_{n},x_{n+1}\right) \right)  &\leq F \left( d\left( x_{n-1},x_{n}\right) \right)-\phi(d\left(x_{n-1},x_{n}\right))\\
			&\leq F \left( d\left( x_{n-2},x_{n-1}\right) \right)-\phi(d\left(x_{n-1},x_{n}\right)) -\phi(d\left(x_{n-2},x_{n-1}\right)) \\
			&\leq ...\leq F\left( d\left( x_{0},x_{1}\right)\right)-\sum_{i=0}^{n}\phi(d\left(x_{i},x_{i+1}\right)). 
		\end{align*}
		Since $ \liminf _{\alpha \rightarrow s^{+} }\phi(\alpha)> 0 $, we have $ \liminf _{n\rightarrow \infty }\phi( d\left( x_{n-1},x_{n}\right))> 0 $, then from the definition of the limit, there exists $ n_0\in \mathbb{N}$  and $ A> 0 $ such that for all $n\geq n_{0}   $, $ \phi( q\left( x_{n-1},x_{n}\right))>A $, hence 
		\begin{align*}
			F\left( d\left( x_{n},x_{n+1}\right) \right)&\leq F\left( d\left( x_{0},x_{1}\right)\right)-\sum_{i=0}^{n_{0}-1}\phi(d\left(x_{i},x_{i+1}\right))-\sum_{i=n_{0}-1}^{n}\phi(d\left(x_{i},x_{i+1}\right))
		\end{align*}
		\begin{align*}
			&\leq F\left( d\left( x_{0},x_{1}\right)\right)-\sum_{i=n_{0}-1}^{n} A
		\end{align*}
		\begin{align}
			&=F\left( d\left( x_{0},x_{1}\right)\right)-(n-n_0)A,  
		\end{align}
		for all  $n\geq n_{0} $.
		Taking limit as $ n\rightarrow\infty $ in above inequality we get
		\begin{equation}
			lim _{n \rightarrow\infty }F\left( d\left( x_{n},x_{n+1}\right) \right)\leq  \lim _{n \rightarrow\infty }\left[ F\left( d\left( x_{0},x_{1}\right) \right)-(n-n_0)A\right],
		\end{equation}
		that is, $lim _{n \rightarrow\infty }F\left( d\left( x_{n},x_{n+1}\right) \right)=-\infty  $ then from the condition $ (ii) $ of Definition $ \ref{2.10} $, we conclude that 
		\begin{equation}\label{3.23}
			\lim_{n\rightarrow \infty }d\left( x_{n,}x_{n+1}\right)=0.
		\end{equation}
		Next. We shall prove that 
		\begin{equation*}
			\lim_{n\rightarrow \infty }d\left( x_{n},x_{n+2}\right) =0.
		\end{equation*}
		We assume that $x_{n}\neq x_{m}$ for every $n,m\in $ $\mathbb{N},\ n\neq m.$ Indeed, suppose that $x_{n}= x_{m}$ for some $  n=m+k$  with $ k>0 $ and using $(\ref{3.19})$, we have  
		\begin{equation}
			d\left( x_{m},x_{m+1}\right) =  d\left( x_{n},x_{n+1}\right) <d\left( x_{n-1},x_{n}\right).
		\end{equation}
		Continuing this process, we can that
		\begin{equation*} 
			d\left( x_{m},x_{n+1}\right)= d\left( x_{n},x_{n+1}\right)<d\left( x_{m},x_{m+1}\right).
		\end{equation*}
		It is a contradiction. Therefore, $d\left( x_{n},x_{m}\right)>0  $ for every $n,m\in $ $\mathbb{N}$, $n\neq m .$
		
		Now, applying $( \ref{3.17})$  with $ x= x_{n-1}$ and $y=x_{n+1}  $, we have
		\begin{equation}
			F \left[ d\left( x_{n},x_{n+2}\right)\right]=F \left[ d\left( Tx_{n-1},Tx_{n+1}\right)\right]\leq F\left[s^{2} d\left( Tx_{n-1},Tx_{n+1}\right)\right]\leq  F \left( M\left(x_{n-1},x_{n+1}\right) \right)-\phi\left(d( x_{n-1},x_{n+1})\right),
		\end{equation}
		where
		\begin{align*}
			M\left( x_{n-1},x_{n+1}\right)&= \max \left\{d\left( x_{n-1},x_{n+1}\right),d\left(x_{n-1},x_{n}\right) ,d\left( x_{n+1},x_{n+2}\right),d\left(x_{n+1},x_{n}\right) \right\}  \\
			&= \max \left\{d\left(x_{n-1},x_{n+1}\right),d\left( x_{n-1},x_{n}\right) \right\}.                       
		\end{align*} 
		So, we get 
		\begin{align}\label{3.26}
			F \left( d\left( x_{n},x_{n+2}\right)\right) &\leq F \left( \max \left\{d\left( x_{n-1},x_{n}\right),d\left( x_{n-1},x_{n+1}\right) \right\} \right)-\phi\left( d(x_{n-1},x_{n+1})\right) 
		\end{align} 
		Take $a_{n}=d\left( x_{n},x_{n+2}\right) $ and $b_{n}=d\left(x_{n},x_{n+1}\right).$ Thus, by $(\ref{3.26}) $, one can write
		\begin{align}\label{3.27}  
			F \left(a_{n}\right)  \leq F \left( \max \left(a_{n-1}, b_{n-1}\right) \right)-\phi\left( d(a_{n-1}\right)).
		\end{align} 
		Since  $F$ is increasing, we get
		\begin{equation*}
			a_{n}<\max \left\{ a_{n-1},b_{n-1}\right\}.
		\end{equation*}
		by $( \ref{3.19})$, we have  
		\begin{equation*}
			b_{n}\leq b_{n-1}\leq \max \left\{ a_{n-1},b_{n-1}\right\}.
		\end{equation*}
		Which implies that
		\begin{equation*}
			\max \left\{ a_{n},b_{n}\right\} \leq \max \left\{ a_{n-1},b_{n-1}\right\},\text{ }\forall n\in \mathbb{N}.
		\end{equation*}
		Therefore, the sequence $ \max \left\{ a_{n-1},b_{n-1}\right\}_{n\in \mathbb{N}}  $  is nonnegative decreasing sequence of real numbers Thus, there exists $\lambda \geq 0$ such that 
		\begin{equation*}
			\lim_{n\rightarrow \infty }\max \left\{ a_{n},b_{n}\right\} =\lambda.
		\end{equation*}
		By $(\ref{3.23}) $ assume that $\lambda >0$, we have
		\begin{equation*}
			\lambda=\lim_{n\rightarrow \infty }\sup a_{n}=\lim_{n\rightarrow \infty }\sup \max\left\{ a_{n},b_{n}\right\} =\lim_{n\rightarrow \infty }\max \left\{a_{n},b_{n}\right\}. 
		\end{equation*}
		Taking the  $\limsup_n\rightarrow \infty $ in  $(\ref{3.27})$ and using the continuity  of $ F $ and the property of $ \phi $, we obtain
		\begin{align*}
			F \left( \lim_{n\rightarrow \infty }\sup  a_{n}\right)&\leq F \left(\lim_{n\rightarrow \infty }\sup\max \left\{  a_{n-1}, b_{n-1}\right\} \right)-\lim_{n\rightarrow \infty }\sup\phi\left(a_{n-1}\right) \\
			&\leq F \left(\lim_{n\rightarrow \infty }\sup\max \left\{  a_{n-1}, b_{n-1}\right\} \right)-\lim_{n\rightarrow \infty }\inf\phi\left( a_{n-1}\right)\\
			&<F \left(\lim_{n\rightarrow \infty }\max \left\{  a_{n-1}, b_{n-1}\right\} \right).
		\end{align*}
		Therefore,
		\begin{equation*}
			F \left( \lambda\right) <F\left(\lambda\right).
		\end{equation*}
		It is a contradiction. Hence, 
		\begin{equation}
			\lim_{n\rightarrow \infty }d\left( x_{n,}x_{n+2}\right) =0.
		\end{equation}
		Next, We shall prove that $\left\lbrace  x_{n}\right\rbrace  _{n\in \mathbb{N}}$ is a Cauchy sequence, i.e, $\lim_{n,m\rightarrow \infty }d\left( x_{n,}x_{m}\right) =0,$ for all $n,m\in \mathbb{N}$.
		Suppose to the contrary. By Lemma $ \ref{Lemma 2.4} $. Then there is an $\varepsilon $ $>0$ such that for an integer $ k $ there exists  two sequences $\left\lbrace n_{\left( k\right) }\right\rbrace $ and $\left\lbrace m_{\left( k\right) }\right\rbrace $
		such that
		\item[i)] $\varepsilon \leq \lim_{k\rightarrow \infty }\inf d\left( x_{m_{\left( k\right) }},x_{n_{\left( k\right)}}\right)  \leq \lim_{k\rightarrow \infty }\sup d\left( x_{m_{\left( k\right) }},x_{n_{\left( k\right)}}\right)\leq s\varepsilon ,$ 
		\item[ii)] $\varepsilon \leq \lim_{k\rightarrow \infty }\inf d\left( x_{n_{\left( k\right) }},x_{m_{\left( k\right)+1}}\right)  \leq \lim_{k\rightarrow \infty }\sup d\left( x_{n_{\left( k\right) }},x_{m_{\left( k\right)+1}}\right)\leq s\varepsilon ,$ 
		\item[iii)] $\varepsilon \leq \lim_{k\rightarrow \infty }\inf d\left( x_{m_{\left( k\right) }},x_{n_{\left( k\right)+1}}\right)  \leq \lim_{k\rightarrow \infty }\sup d\left( x_{m_{\left( k\right) }},x_{n_{\left( k\right)+1}}\right)\leq s\varepsilon ,$ 
		\item[vi)] $\frac{\varepsilon}{s} \leq \lim_{k\rightarrow \infty }\inf d\left( x_{m_{\left( k\right)+1 }},x_{n_{\left( k\right)+1}}\right)  \leq \lim_{k\rightarrow \infty }\sup d\left( x_{m_{\left( k\right)+1 }},x_{n_{\left( k\right)+1}}\right)\leq s^{2}\varepsilon .$\\
		From $ (\ref{3.17}) $ and by setting $ x = x_{m_{\left( k\right) }} $ and $ y = x_{n_{\left( k\right) }} $ we have:
		\begin{align}
			\lim_{k\rightarrow \infty }M\left( x_{m_{\left( k\right) }},x_{n_{\left( k\right) }}\right)& =\lim_{k\rightarrow \infty }\max\left\{d\left( x_{m_{\left( k\right) }},x_{n_{\left( k\right)}}\right),d\left( x_{m_{\left( k\right) }},x_{{m\left( k\right)
					+1}}\right) ,d\left( x_{n_{\left( k\right) }},x_{n_{\left(k\right) +1}}\right),d\left( x_{n_{\left( k\right) }},x_{m_{\left( k\right)+1}}\right)\right\}
			& \leq s \varepsilon.
		\end{align}
		Now, applying $ (\ref{3.17}) $ with $ x=x_{m_{\left( k\right) }} $ and  $ y=x_{n_{\left( k\right) }}$, we obtain
		\begin{equation}
			F \left[ s^{2}d\left( x_{m_{\left( k\right) +1}},x_{n_{\left( k\right)+1}}\right)\right]  \leq  F \left( M\left( x_{m_{\left( k\right)}},x_{n_{\left( k\right) }}\right) \right)-\phi\left( d(x_{m_{\left( k\right)}},x_{n_{\left( k\right) }})\right) .
		\end{equation}
		Letting $ k\rightarrow\infty $ the above inequality and using $(3.29)$ and $(vi)$, we obtain 
		\begin{align*}
			F\left( \frac{\varepsilon}{s}s^{2}\right)& =F\left( \varepsilon s\right)\\
			& \leq F \left(s^{2}\lim_{k\rightarrow \infty }\sup d\left( x_{m_{\left( k\right)+1}},x_{n_{\left( k\right) +1}}\right)\right)\\
			&=\lim_{k\rightarrow \infty }\sup F \left(s^{2} d\left( x_{m_{\left( k\right)+1}},x_{n_{\left( k\right) +1}}\right)\right)\\
			& \leq \lim_{k\rightarrow \infty }\sup F \left ( M\left( x_{m_{\left( k\right)}},x_{n_{\left( k\right) }}\right) \right)-\lim_{k\rightarrow \infty }\sup\phi\left(d( x_{m_{\left( k\right)}},x_{n_{\left( k\right) }})\right)\\
			&= F \left (  \lim_{k\rightarrow \infty }\sup M\left( x_{m_{\left( k\right)}},x_{n_{\left( k\right) }}\right) \right)-\lim_{k\rightarrow \infty }\sup\phi\left(d( x_{m_{\left( k\right)}},x_{n_{\left( k\right) }})\right)\\
			& \leq  F \left (\lim_{k\rightarrow \infty }\sup M\left( x_{m_{\left( k\right)}},x_{n_{\left( k\right) }}\right) \right)-\lim_{k\rightarrow \infty }\inf \phi\left( d(x_{m_{\left( k\right)}},x_{n_{\left( k\right) }})\right)\\
			&< F \left (\lim_{k\rightarrow \infty }\sup M\left( x_{m_{\left( k\right)}},x_{n_{\left( k\right) }}\right) \right)\\
			& \leq F \left (s \varepsilon \right).
		\end{align*} 
		Therefore, $$F(s\varepsilon) <F(s\varepsilon). $$
		Since $ F $ is increasing, we get
		\begin{equation*}
			s \varepsilon  <s \varepsilon.
		\end{equation*}
		It is a contradiction. Then 
		\begin{equation*}
			\lim_{n,m\rightarrow \infty }d\left( x_{m},x_{n}\right) =0.
		\end{equation*}
		Hence $\left\lbrace  x_{n}\right\rbrace  $ is a Cauchy sequence in X. By completeness of $\left( X,d\right) ,$ there exists $z\in X$ such that 
		\begin{equation*}
			\lim_{n\rightarrow \infty }d\left( x_{n},z\right)  =0. 
		\end{equation*}
		Now, we show that $d\left( Tz,z\right) =0$  arguing by contradiction, we assume that
		\begin{equation*}
			d\left( Tz,z\right)>0.
		\end{equation*} 
		Since $ x_{n}\rightarrow z $ as $ n\rightarrow \infty $ for all $ n\in \mathbf{N} $, then from Lemma $ \ref{Lemma 2.3} $, we conclude that
		\begin{equation}\label{3.31}
			\frac{1}{s}d\left( z,Tz\right)\leq \lim_{n\rightarrow \infty }\sup d\left(Tx_{n},Tz\right) \leq sd\left( z,Tz\right).
		\end{equation}
		Now, applying $(\ref{3.17}) $ with $ x=x_n $ and $ y=z $, we have
		\begin{equation*}
			F \left( s^{2}d\left( Tx_{n},Tz\right)\right) \leq F\left( M\left( x_{n},z\right) \right)-\phi\left(d( x_{n},z)\right) ,\text{ }\forall n\in \mathbb{N},
		\end{equation*}
		where 
		\begin{equation*}
			M\left(x_{n}, z\right) =\max \left\{ d\left( x_{n},z\right) ,d\left( x_{n},Tx_{n}\right) ,d\left( z,Tz\right),d\left( z,Tx_{n}\right)\right\}.
		\end{equation*}
		and 
		\begin{equation}\label{3.32}
			\lim_{n\rightarrow  \infty }\sup\max \left\{ d\left( x_{n},z\right) ,d\left( x_{n},Tx_{n}\right) ,d\left( z,Tz\right),d\left( z,Tx_{n}\right)\right\}=d\left( z,Tz\right).
		\end{equation}
		Therefore,
		\begin{equation}\label{3.33}
			F \left( s^{2}d\left(Tx_{n} ,Tz \right)\right) \leq F\left( \max \left\{ d\left( x_{n},z\right) ,d\left( x_{n},Tx_{n}\right) ,d\left( z,Tz\right),d\left( z,Tx_{n}\right)\right\}\right)-\phi\left(d( x_{n},z)\right). 
		\end{equation}
		By letting $n\rightarrow \infty $ in inequality $(\ref{3.33}) $,  using $(\ref{3.31}) $, $(\ref{3.32}) $ and continuity of $F$ we obtain
		\begin{align*}
			F\left[ s^{2}\frac{1}{s}d\left( z,Tz\right)\right]& =F\left[s d\left( z,Tz\right)\right] \\
			&\leq F\left[ s^{2}\lim_{n\rightarrow  \infty }\sup d\left(Tx_{n},Tz\right)\right] \\
			&= \lim_{n\rightarrow  \infty }\sup F\left[ s^{2} d\left(Tx_{n},Tz\right)\right] \\
			& \leq \lim_{n\rightarrow  \infty }\sup F \left( M\left( x_{n},z\right)\right)-\lim_{n\rightarrow  \infty }\phi\left(d( x_{n},z)\right)\\
			&= F \left( d\left(Tz,z\right)\right)-\lim_{n\rightarrow  \infty }\phi\left(d( x_{n},z)\right)\\
			&<F \left( d\left( z,Tz\right)\right).
		\end{align*} 
		Since $ F $ is increasing, we get
		$$ sd (z, Tz) < d (z, Tz)$$
		Which implies that
		$$ d (z,Tz) (s - 1) < 0 \Rightarrow  s < 1.$$
		Which is a contradiction. Hence $Tz=z$.
		
		Uniqueness. Now, suppose that $z,u\in X$ are two fixed points of $T$ such that $u\neq z$. Therefore, we have
		\begin{equation*}
			d\left( z,u\right) =d\left( Tz,Tu\right)  >0.
		\end{equation*}
		Applying $(\ref{3.17}) $ with $ x=z $ and $ y=u $, we have
		\begin{equation*}
			F \left( d\left( z,u\right)\right)=F \left( d\left( Tu,Tz\right)\right)\leq F \left(s^{2} d\left( Tu,Tz\right)\right) \leq F\left( M\left( z,u\right)\right)-\phi\left(d(z,u )\right), 
		\end{equation*}
		where
		\begin{equation*}
			M\left( z,u\right) =\max \left\{d\left( z,u\right),d\left(z,Tz\right),d\left( u,Tu\right),d\left( u,Tz\right)\right\} =d\left(z,u\right).
		\end{equation*}
		Therefore, we have 
		\begin{equation*}
			F \left( d\left( z,u\right)\right)\leq F\left(d\left( z,u\right)\right) -\phi\left(d(z,u) \right) \\
			<F \left(d\left( z,u\right)\right).
		\end{equation*}
		Which implies that
		\begin{equation*}
			d\left( z,u\right) <d\left( z,u\right).
		\end{equation*}
		It is a contradiction. Therefore $u=z$.
	\end{proof}
	\begin{corollary}
		Let $\left( X,d\right) $ be a complete b-rectangular metric space and $T:X\rightarrow X$ $\ $be given mapping. Suppose that there exist $F \in \Im$ and $\tau\in \left] 0,\infty\right[ $ such that for any $x,y\in X,$ we have
		$$d\left( Tx,Ty\right) >0\Rightarrow F\left[ s^{2}d\left( Tx,Ty\right)\right]+\tau \leq \left[F \left( M\left( x,y\right)\right) \right],$$
		where 
		$$M(x,y) =\max\lbrace d\left( x,y\right),d\left( x,Tx\right),d\left( y,Ty\right),d\left( Tx,y\right)\rbrace. $$
		Then $T$ has a unique fixed point.
	\end{corollary}
	It follows from Theorem $ \ref{3.5} $, we obtain the follows fixed point theorems for $F-\phi-$Kannan-type contraction and $F -\phi $- Reich-type contraction.
	\begin{theorem}
		Let $(X,d)$ be a complete b-rectangular metric space and $T:X\rightarrow X $ be a Kannan-type contraction, then $T$ has a unique fixed.
	\end{theorem}
	\begin{proof}
		Since $T $ is a - Kannan-type contraction. Then there exist exist $F \in \Im$ and $\phi \in \Phi $
		such that 
		\begin{align*}
			F \left[s ^{2}d\left( Tx,Ty\right) \right]+\phi\left(d(x,y) \right) & \leq F \left( \frac{d\left( Tx,x\right)  +d\left( Ty,y\right) }{2}\right) \\
			& \leq F\left(\max \left\{d\left( x,Tx\right),d\left( y,Ty\right) \right\}\right) \\
			& \leq F\left(\max \left\{d(x,y),d\left( x,Tx\right) ,d\left( y,Ty\right),d\left( y,Tx\right) \right\}\right).
		\end{align*}
		Therefore, T is  $F -\phi -$contraction. As in the proof of Theorem $ \ref{3.5} $ we conclude that T has a unique fixed point.
	\end{proof}
	\begin{theorem}
		Let $(X,d)$ be a complete b-rectangular metric space and $T:X\rightarrow X$ be a Reich-type contraction. Then T has a unique fixed point.
	\end{theorem}
	\begin{proof}
		Since $T $ is a Reich-type contraction. Then there exist exist $F \in\Im $ and $\phi \in \Phi $ such that 
		\begin{align*}
			F \left[s ^{2} d\left( Tx,Ty\right)\right]+\phi\left(d(x,y) \right)& \leq F \left( \frac{d\left( x,y\right) +d\left( Tx,x\right)  +d\left( Ty,y\right) }{3}\right)\\
			& \leq F\left(\max \left\{d(x,y),d\left( x,Tx\right),d\left( y,Ty\right),d\left( y,Tx\right) \right\}\right).
		\end{align*}
		Therefore, $T$ is  $F -\phi -$contraction. As in the proof of Theorem $ \ref{3.5} $ we conclude that $ T $ has a unique fixed point.
	\end{proof}
	Very recently, Kari et al in \cite{KARO} proved the result  (Theorem 1)   on $ \left( \alpha,\eta\right)  $-complete rectangular b-metric spaces. In this paper, we prove this result in complete rectangular b-metric spaces. 
	\begin{corollary}
		Let $d\left( X,d\right) $ be a complete b-rectangular metric space with parameter $s>1$ and let $T$ be a self mapping on $\ X$. If for all $x,y\in X$ we have 
		\begin{equation}
			d\left( Tx,Ty\right) >0\Rightarrow F \left( s^{2}d\left( Tx,Ty\right) \right)+\phi\left(d(x,y) \right) \leq F \left( \beta_{1}d\left( x,y\right) +\beta _{2}d\left( Tx,x\right) +\beta _{3}d\left(Ty,y\right) +\beta _{4}d\left( y,Tx\right) \right)
		\end{equation}
		where $F \in \mathbb{F} ,\ \phi \in \Phi$, $\beta _{i}\geq0$ for $i\in\{1,2,3,4\},$ $\sum\limits_{\substack{i=0 }}^{i=4} {\beta_i}\leq1$.
		Then $T$ has a unique fixed point.
	\end{corollary}
	\begin{proof}
		We prove that T is a $ F-\phi-$contraction. Indeed,
		\begin{align*}
			F \left( s^{2}d\left( Tx,Ty\right) \right)+\phi\left(d(x,y )\right) &\leq F \left( \beta_{1}d\left( x,y\right) +\beta _{2}d\left( Tx,x\right) +\beta _{3}d\left(Ty,y\right) +\beta _{4}d\left( y,Tx\right) \right)\\
			&\leq F \left( \beta_{1}+\beta _{2}+\beta _{3}+\beta _{4}\right) \left( \max\lbrace d\left( x,y\right) ,d\left( Tx,x\right), d\left(Ty,y\right), d\left( y,Tx\right)\rbrace \right)\\
			&\leq F \left( \max\lbrace d\left( x,y\right) ,d\left( Tx,x\right), d\left(Ty,y\right), d\left( y,Tx\right)\rbrace \right). 
		\end{align*}
		As in the proof of Theorem $ \ref{3.5} $, $T$ has a unique fixed point.
	\end{proof}
	\begin{example}
		Let $ X=A\cup B $, where $ A=\lbrace 0,\frac{1}{2},\frac{1}{3},\frac{1}{4}\rbrace $ and $ B=\left[1,\frac{5}{2} \right]  $.
		
		Define $ d:X\times X\rightarrow \left[0,+\infty \right[  $ as follows:
		\begin{equation*}
			\left\lbrace
			\begin{aligned}
				d(x, y) &=d(y, x)\ for \ all \  x,y\in X;\\
				d(x, y) &=0\Leftrightarrow y= x.\\	
			\end{aligned}
			\right.
		\end{equation*}
		and
		\begin{equation*}
			\left\lbrace
			\begin{aligned}		    
				d\left( 0,\frac{1}{2}\right) =d\left( \frac{1}{2},\frac{1}{3}\right)&=0,16\\
				d\left(0,\frac{1}{3}\right) =d\left( \frac{1}{3},\frac{1}{4}\right)	&=0,04\\
				d\left(0,\frac{1}{4}\right) =d\left( \frac{1}{2},\frac{1}{4}\right)	&=0,25\\
				d\left( x,y\right) =\left( \vert x-y\vert\right) ^{2} \ otherwise.
			\end{aligned}
			\right.
		\end{equation*}
		Then $ (X,d) $ is a b-rectangular metric space with coefficient s=3.
		
		Define mapping $T:X\rightarrow X$ by
		\begin{equation*}
			T(x)=\left\lbrace
			\begin{aligned}
				x^{\frac{1}{6}}	& \ if \ x\in \left[1,\frac{5}{2} \right]\\
				1&  \ if \ x\in A.
			\end{aligned}
			\right.
		\end{equation*}
		Evidently, $ T(x)\in X $. Let $F( t) =ln(\sqrt{t})$,  $\phi (t)=\frac{1}{2+t}$. It obvious that $F \in \Im$ and $\phi\in\Phi .$
		
		Consider the following possibilities:
		\item[1] : $ x,y \in  \left[1,\frac{5}{2} \right]  $. Then
		$$ T(x)=x^{\frac{1}{6}}	,\   T(y)=y^{\frac{1}{6}},\ d\left( Tx,Ty \right)=\left( x^{\frac{1}{6}}-y^{\frac{1}{6}}\right) ^{2},\ d(x,y)=(x-y)^{2}.  $$
		On the other hand 
		$$ F \left[s^{2}d\left(Tx,Ty\right) \right]=ln(3(x^{\frac{1}{6}}-y^{\frac{1}{6}})), $$
		$$ F \left[d\left(x,y\right)\right]=ln\left[ (x-y)\right]  $$
		and
		$$ \phi \left[d\left(x,y\right)\right]=\frac{1}{2+(x-y)^{2} } .$$
		We have 
		\begin{align*}
			F \left[s^{2}d\left(Tx,Ty\right) \right]+\phi \left[d\left(x,y\right)\right]-F \left[d\left(x,y\right)\right]&=ln(3(x^{\frac{1}{6}}-y^{\frac{1}{6}}))-ln((x-y)) +\frac{1}{2+(x-y)^{2} }\\
			&=-ln\left( \frac{(x^{\frac{2}{3}} +y^{\frac{2}{3}}+x^{\frac{1}{3}}y^{\frac{1}{3}})}{3}\right)-ln\left(x^{\frac{1}{6}} +y^{\frac{1}{6}}) \right)+\frac{1}{2+(x-y)^{2} } 
		\end{align*}
		Since  $ x,y \in  \left[1,\frac{5}{2} \right]  $, then
		$$  \frac{(x^{\frac{2}{3}} +y^{\frac{2}{3}}+x^{\frac{1}{3}}y^{\frac{1}{3}})}{3}\geq 1\Rightarrow -ln\left( \frac{(x^{\frac{2}{3}} +y^{\frac{2}{3}}+x^{\frac{1}{3}}y^{\frac{1}{3}})}{3}\right)\leq 0  ,$$
		and
		$$-ln\left(x^{\frac{1}{6}}+y^{\frac{1}{6}} \right) +\frac{1}{2+(x-y)^{2} }\leq -ln\left(x^{\frac{1}{2}}+y^{\frac{1}{2}} \right) +\frac{1}{2}\leq 0.     $$ 
		Thus, for all $ x,y\in \left[1,\frac{5}{2}\right] $ with $ x\neq y $, we have
		$$  F \left[s^{2}d\left(Tx,Ty\right) \right]+\phi \left[d\left(x,y\right)\right]\leq F \left[d\left(x,y\right)\right]\leq  F \left[M\left(x,y\right)\right]$$
		Case $ 2 $ : $ x \in  \left[1,2 \right]  $, $ y\in A $. Then
		$$ T(x)=x^{\frac{1}{6}},\   T(y)=1	,\ d\left( Tx,Ty \right)=\left( x^{\frac{1}{6}}-1\right) ^{2},\ d(x,y)=(x-y)^{2}.  $$
		On the other hand 
		$$ F \left[s^{2}d\left(Tx,Ty\right) \right]=ln(3(x^{\frac{1}{6}}-1)),   $$
		$$ F \left[d\left(x,y\right)\right]=ln(x-y) $$
		and
		$$ \phi \left[d\left(x,y\right)\right]=\frac{1}{2+(x-y)^{2} } .$$
		We have 
		\begin{align*}
			F \left[d\left(x,y\right)\right]- F \left[s^{2}d\left(Tx,Ty\right) \right]-\phi \left[d\left(x,y\right)\right]&= ln(x-y)- \frac{1}{2+(x-y)^{2} }-ln(3(x^{\frac{1}{6}}-1)\\
			&=ln\left[ \frac{(x-y) }{3(x^{\frac{1}{6}}-1)}\right]- \frac{1}{2+(x-y)^{2} }.
		\end{align*}
		Since  $ x \in  \left[1,\frac{5}{2} \right]  $ and $ y\in A $, then
		$$(x-y)^{2}\geq \left( x-\frac{1}{2}\right) ^{2}=\left( x-1+\frac{1}{2}\right) ^{2}>\left( x-1\right) ^{2}  .$$
		Hence 
		$$(x-y)>\left( x-1\right)=\left( x^{\frac{1}{6}}-1 \right)\left( x^{\frac{1}{6}}+1 \right) \left( x^{\frac{2}{3}}+x^{\frac{1}{3}}+1 \right),$$
		so,
		$$\frac{(x-y)}{\left( x^{\frac{1}{6}}-1 \right)}>\left( x^{\frac{1}{6}}+1 \right) \left( x^{\frac{2}{3}}++x^{\frac{1}{3}}+1 \right) $$
		and we have 
		$$ln\left[ \frac{x-y }{3(x^{\frac{1}{6}}-1)}\right]>  ln\left[ \frac{\left( x^{\frac{1}{6}}+1 \right) \left( x^{\frac{2}{3}}+x^{\frac{1}{3}}+1 \right) }{3}\right] = ln\left[ \frac{\left( x^{\frac{2}{3}}+x^{\frac{1}{3}}+1 \right)  }{3}\right]+ ln\left[  \left( x^{\frac{1}{6}}+1 \right)\right].$$
		Since $ x\in\left[1,\frac{5}{2} \right]  $, then 
		$$  ln\left[ \frac{\left( x^{\frac{2}{3}}+x^{\frac{1}{3}}+1 \right)  }{3}\right]\geq 0 \ , \ ln\left[  \left( x^{\frac{1}{6}}+1 \right)\right]\geq ln(2), \frac{1}{2+(x-y)^{2} }\leq \frac{1}{2}\leq ln(2).$$
		Hence
		$$ ln\left[  \left( x^{\frac{1}{6}}+1 \right)\right]\geq  \frac{1}{2+(x-y)^{2} } .$$
		Thus, for all $ x\in \left[1,\frac{2}{5} \right] $ and $ y\in A  $, we have
		$$  F \left[sd\left(Tx,Ty\right) \right]+\phi \left[d\left(x,y\right)\right]\leq F \left[d\left(x,y\right)\right] $$
		Hence, the condition $ (\ref{3.17}) $ is satisfied. Therefore, $ T $ has a unique fixed point $ z=1 $.  
	\end{example}
	In this section, we apply our results to solve the following nonlinear integral equation problem:
	\section{Application to nonlinear integral equations}
	\begin{equation}\label{4.1}	
		x(t)=\lambda\int_a^bK(t,r,x(r))ds,
	\end{equation}
	where $ a,b\in\mathbb{R} $, $ x\in C(\left[a,b \right],\mathbb{R})$ and $ K:\left[a,b \right]^{2}\times \mathbb{R}\rightarrow \mathbb{R} $ is a given continuous function. 
	\begin{theorem} \label{theorem 4.2}
		Consider the nonlinear integral equation problem: $ (4.1) $ and assume that the kernel function $ K $ satisfies the condition 
		$\vert K(t,r,x(r)) -  K(t,r,y(r))\vert \leq \frac{1}{s^{2+s}} e^{-\frac{1}{\vert x(t)-  y(t)\vert+1}}\left( \vert x(t)-  y(t)\vert\right) $
		for all $t,r \in \left[a,b \right]  $ and $x,y \in\mathbb{R}$. Then the equation $ (4.1) $ has a unique solution $ x\in C(\left[a,b \right]  $ for some constant $ \lambda $ depending on the constants $ a,b\ and\ s $. 
	\end{theorem}
	\begin{proof}
		Let $ X =C(\left[a,b \right]$ and $ T:X\rightarrow X $ defined by 
		$$ T(x)(t)=\lambda\int_a^bK(t,r,x(r))ds,  $$
		for all $ x\in X $. Clearly, $ X $ with the metric $ d: X\times X\rightarrow \left[0,+\infty \right[  $ given by
		$$ d(x,y)=\left( \max_{t\in\left[a,b \right] }\vert x(t)- y(t)\vert\right)^{s} , $$ for all $ x,y\in X $. It is clear that $ X,d $ is a complete $ b $-rectangular metric space. 
		
		We will find the condition on $ \lambda $ under which the operator has a unique fixed point which will the solution of the integral equation $ (4.1) $. Assume that, $ x,y\in X $ and $ t,r \in \left[a,b \right]$. Then we get
		\begin{align*}
			\vert Tx(t)-  Ty(t)\vert ^{s}&=\vert\lambda\vert ^{s}\left( \vert\int_a^bK(t,r,x(r))dr-\int_a^bK(t,r,y(r))dr\vert\right)^{s} \\
			&= \vert\lambda ^{s}\vert \vert \int_a^b K(t,r,x(r))- K(t,r,y(r))dr\vert ^{s}\\
			&\leq \vert\lambda\vert ^{s}\int_a^b\vert K(t,r,x(r))- K(t,r,y(r))dr\vert ^{s}\\
			& \leq \vert\lambda\vert ^{s} \int_a^b\left( \frac{1}{s^{2+s}} e^{-\frac{1}{\vert x(r)-  y(r)\vert+1}}\left( \vert x(r)\vert- \vert y(r)\vert\right)dr \right) ^{s}\\
			&=\frac{1}{s^{2}}\vert\lambda \vert ^{s}  \int_a^b\left(  e^{-\frac{1}{\vert x(r)-  y(r)\vert+1}}\left( \vert x(r)\vert- \vert y(r)\vert\right)\right) dr.
		\end{align*}
		Which implies that 
		\begin{align*}
			\max_{t\in\left[a,b \right] }\left( \vert Tx(t)- Ty(t)\vert\right) &=\max_{t\in\left[a,b \right] } \vert\lambda\vert ^{s}\int_a^b\vert K(t,r,x(r))- K(t,r,y(r))dr\vert ^{s}\\
			&\leq\max_{t\in\left[a,b \right] }\frac{1}{s^{2}}\vert\lambda \vert ^{s}  \int_a^b\left(  e^{-\frac{1}{\vert x(r)-  y(r)\vert+1}}\left( \vert x(r)-  y(r)\vert\right)dr\right)  ^{s}\\
			& \leq \vert\lambda\vert ^{s} \frac{1}{s^{2}} \int_a^b\left( e^{-\max_{s\in\left[a,b \right] }\frac{1}{\vert x(r)-  y(r)\vert+1}}\left( \max_{r\in\left[a,b \right] }\vert x(r)-  y(r)\vert\right)dr \right)  ^{s}.
		\end{align*}
		Since by the definition of the $ b $-rectangular metric space, we have $ d(Tx,Ty) > 0 $ and $ d(x,y) > 0 $ for any $ x\neq y $, then we can take natural logarithm sides and get 
		\begin{align*}
			ln\left[ s^{2}d(Tx,Ty)\right]&=ln\left[s^{2}\vert\lambda\vert ^{s}\max_{t\in\left[a,b \right] } \int_a^b\vert K(t,r,x(r))- K(t,r,y(r))dr\vert ^{s}\right]\\
			&\leq ln\left[ \vert\lambda\vert ^{s}  \int_a^b\left( e^{-\max_{r\in\left[a,b \right] }\frac{1}{\vert x(r) - y(r)\vert+1}}\left( \max_{r\in\left[a,b \right] }\vert x(r)-  y(r)\vert\right)dr \right)  ^{s}\right]\\
			&=ln\left[\left( (b-a) \vert\lambda\vert \right) ^{s}\right] + ln\left[   \int_a^b\left( e^{-\max_{r\in\left[a,b \right] }\frac{1}{\vert x(r)-  y(r)\vert+1}}\left( \max_{r\in\left[a,b \right] }\vert x(r)-  y(r)\vert\right)dr \right) \right] ^{s}\\
			&=ln\left[\left( (b-a) \vert\lambda\vert \right) ^{s}\right] +ln\left(\left(  e^{-\max_{t\in\left[a,b \right] }\frac{1}{\vert x(t)-  y(t)\vert+1}} \right)^{s}\right)  + ln\left[   \int_a^b\left( \left( \max_{r\in\left[a,b \right] }\vert x(r)- y(r)\vert\right)dr\right)  ^{s}\right]\\
			&=s.ln\left[(b-a) \vert\lambda\vert \right]- \frac{s}{\max_{t\in\left[a,b \right] }\vert x(t)-  y(t)\vert+1}  + ln\left[   \int_a^b\left( \left( \max_{r\in\left[a,b \right] }\vert x(r)-  y(r)\vert\right)dr \right)  ^{s}\right]
		\end{align*}
		provided that $  \vert\lambda\vert (b-a)\leq \frac{1}{e^{s}} $, which implies that
		\begin{align*}
			ln\left[ s^{2}d(Tx,Ty)\right]&\leq -\frac{s}{d(x,y)+1}+ln\left( d(x,y)\right)\\
			&\leq -\frac{1}{d(x,y)+1}+ln\left( d(x,y)\right). 
		\end{align*}
		Hence
		\begin{align}
			F\left(  s^{2}d(Tx,Ty)\right) + \phi(d(x,y))\leq F\left( d(x,y)\right),
		\end{align}
		for all x,y $ \in X $ with $ F(t)=ln(t) $ and $ \phi(t)=\frac{1}{t+1} $. It follows that $ T $ satisfies the condition $( \ref{3.17}) $. Therefore there exists a unique solution of the nonlinear Fredholm inequality $ (\ref{4.1}) $.
	\end{proof}
	\bibliographystyle{amsplain}
	
\end{document}